\documentclass[11pt]{article} 
\usepackage[utf8]{inputenc} 

\usepackage[margin=1in]{geometry} 
\usepackage{graphicx} 
\usepackage{dsfont}
\usepackage{booktabs} 
\usepackage{array} 
\usepackage{paralist} 
\usepackage{verbatim} 
\usepackage{mathrsfs}
\usepackage{amssymb}
\usepackage{amsthm}
\usepackage{amsmath,amsfonts,amssymb}
\usepackage{esint}
\usepackage{graphics}
\usepackage{enumerate}
\usepackage{mathtools}
\usepackage{xfrac}
\usepackage{subcaption}
\usepackage{stmaryrd}
\usepackage{mathabx}
\usepackage{caption}
\usepackage{amsmath}
\usepackage{graphicx}

\usepackage[dvipsnames]{xcolor}
\usepackage[colorlinks=true, pdfstartview=FitV, linkcolor=blue, citecolor=blue, urlcolor=blue]{hyperref}
\usepackage[normalem]{ulem}

\usepackage{tikz}
\usetikzlibrary{calc}
\usepackage{pgf}

\numberwithin{equation}{section}
\numberwithin{figure}{section}

\newtheorem{theorem}{Theorem}[section]

\newtheorem{corollary}[theorem]{Corollary}
\newtheorem{proposition}[theorem]{Proposition}
\newtheorem{lemma}[theorem]{Lemma}
\theoremstyle{definition}
\newtheorem{definition}[theorem]{Definition}

\newtheorem{remark}[theorem]{Remark}

\DeclarePairedDelimiter{\floor}{\lfloor}{\rfloor}

\newcommand*{\R}{\ensuremath{\mathbb{R}}}

\newcommand{\eps}{\varepsilon}

\renewcommand*{\tilde}{\widetilde}

\renewcommand{\P}{\ensuremath{\mathbb{P}}}

\DeclareMathOperator{\dist}{dist}

\DeclareSymbolFont{boldoperators}{OT1}{cmr}{bx}{n}
\SetSymbolFont{boldoperators}{bold}{OT1}{cmr}{bx}{n}
\usepackage{accents}

\newcommand{\T}{\mathbb{T}}

\def\XXint#1#2#3{{\setbox0=\hbox{$#1{#2#3}{\int}$}
\vcenter{\hbox{$#2#3$}}\kern-.5\wd0}}

\let\originalleft\left
\let\originalright\right
\renewcommand{\left}{\mathopen{}\mathclose\bgroup\originalleft}
\renewcommand{\right}{\aftergroup\egroup\originalright}

\newcommand{\sol}{\mathcal{T}}

\newcommand{\indc}{{\mathds{1}}}

\newcommand{\E}{\mathbb{E}}

\makeatletter
\pgfmathdeclarefunction{erf}{1}{%
  \begingroup
    \pgfmathparse{#1 > 0 ? 1 : -1}%
    \edef\sign{\pgfmathresult}%
    \pgfmathparse{abs(#1)}%
    \edef\x{\pgfmathresult}%
    \pgfmathparse{1/(1+0.3275911*\x)}%
    \edef\t{\pgfmathresult}%
    \pgfmathparse{%
      1 - (((((1.061405429*\t -1.453152027)*\t) + 1.421413741)*\t 
      -0.284496736)*\t + 0.254829592)*\t*exp(-(\x*\x))}%
    \edef\y{\pgfmathresult}%
    \pgfmathparse{(\sign)*\y}%
    \pgfmath@smuggleone\pgfmathresult%
  \endgroup
}
\makeatother

\usepackage{titlesec}

\newcommand{\addperiod}[1]{#1.}
\titleformat{\section}
   {\centering\normalfont\Large}{\thesection.}{0.5em}{}
\titleformat*{\subsection}{\bfseries}
\titleformat{\subsubsection}[runin]
  {\normalfont\bfseries}
  {\thesubsubsection.}
  {0.5em}
  {\addperiod}
\titleformat*{\subsubsection}{\normalfont\itshape}
\titleformat*{\paragraph}{\bfseries}
\titleformat*{\subparagraph}{\large\bfseries}

\title{A universal total anomalous dissipator}

\author{Elias Hess-Childs\thanks{Department of Mathematical Sciences, Carnegie Mellon University.
{\footnotesize \href{mailto:aa@cims.nyu.edu}{ehesschi@andrew.cmu.edu}.}
}
\and 
Keefer Rowan\thanks{Courant Institute of Mathematical Sciences,  New York University.
{\footnotesize \href{mailto:keefer.rowan@cims.nyu.edu}{keefer.rowan@cims.nyu.edu}.}
}
}
\date{\today}

\usepackage[nottoc,notlot,notlof]{tocbibind}

\begin{document}

\maketitle

\begin{abstract}
    For all $\alpha\in(0,1)$, we construct an explicit divergence-free vector field $V\in L^\infty_tC^\alpha_x \cap C^{\frac{\alpha}{1-\alpha}}_t L^\infty_x$ so that the solutions to the drift-diffusion equations
    \[\partial_t\theta^\kappa-\kappa\Delta\theta^\kappa+V\cdot\nabla\theta^\kappa=0\]
    exhibit asymptotic total dissipation for all mean-zero initial data: $\lim_{\kappa\rightarrow 0}\|\theta^\kappa(1,\cdot)\|_{L^2}=0$. Additionally, we give explicit rates in $\kappa$ and uniform dependence on initial data.
\end{abstract}

\section{Introduction}

In this paper, we consider solutions to the drift-diffusion equation
\begin{equation}\label{eq:intro_drift_diffusion_equation}
        \begin{cases}
            \partial_t \theta^\kappa - \kappa \Delta \theta^\kappa + V \cdot \nabla \theta^\kappa = 0&\text{in}\ (0,1)\times \T^2,\\
            \theta^\kappa(0,\cdot) = \theta_0(\cdot)&\text{on}\ \T^2,
        \end{cases}
\end{equation}
where $V(t,x)$ is a divergence-free vector field and $\theta_0$ is in $TV(\T^2),$ the space of Borel measures with finite total variation.

The main result of this paper is the existence of an explicit vector field $V(t,x)$ so that the corresponding passive scalar $\theta^\kappa$ exhibits \textit{asymptotic total dissipation} for all initial data.

\begin{theorem}\label{thm:total_dissipation}
    For all $\alpha \in (0,1)$, there exists a constant $C(\alpha)>0$ and a divergence-free vector field $V\in L^\infty([0,1], C^\alpha(\T^2)) \cap C^{\frac{\alpha}{1-\alpha}}([0,1], L^\infty(\T^2))$ so that for all mean-zero $\theta_0 \in TV(\T^2)$ and $\kappa>0$, the unique solution $\theta^\kappa$ to~\eqref{eq:intro_drift_diffusion_equation} satisfies the estimate
    \begin{equation}
    \label{eq:main-theorem-estimate}
    \|\theta^\kappa(1,\cdot)\|_{L^1(\T^2)} \leq C\kappa^{\frac{(1-\alpha)^2}{72}}\|\theta_0\|_{TV(\T^2)}.
    \end{equation}
\end{theorem}


In the above statement and throughout, for $a \in [0,\infty)$, we use $C^a$ to denote the H\"older space $C^{\floor{a}, a- \floor{a}}$.

As a simple corollary, we get asymptotic total dissipation estimates in $L^p$ for all $p \in [1,\infty]$ as well as asymptotic total dissipation in the positive regularity norm $H^\sigma$ for negative regularity data in $H^{-\sigma}.$ We use the vector field $V$ of Theorem~\ref{thm:total_dissipation} first forward in time and then backward in time. This allows us to get dissipation in $L^1$ from Theorem~\ref{thm:total_dissipation} and then dissipation in $L^\infty$ from taking the adjoint. Additionally, we add times during which the vector field is $0$ to exploit the smoothing of the heat equation. We delay the proof until the end of Section~\ref{sec:definition-outline}. 

\begin{definition}
\label{def:W-def}
    Let $V : [0,1] \times \T^2 \to \R^2$ be as in Theorem~\ref{thm:total_dissipation}. Then we define $W : [0,1] \times \T^2 \to\R^2,$
    \[W(t,x) := V(4t -1, x) \indc_{t \in [1/4,1/2]} + V(3-4t, x) \indc_{t \in [1/2,3/4]}.\]
\end{definition}

\begin{corollary}
\label{cor:other-dissipations}
    For all $\alpha \in (0,1)$, there exists a constant $C(\alpha)>0$ so that for all mean-zero $\theta_0$ and $\kappa>0$, the unique solution $\theta^\kappa$ to~\eqref{eq:intro_drift_diffusion_equation}---with $W$ of Definition~\ref{def:W-def} in place of $V$---satisfies the estimates
    \begin{align}
    \label{eq:cor-lp}
        \|\theta^\kappa(1,\cdot)\|_{L^p(\T^2)}&\leq C \kappa^{\frac{(1-\alpha)^2}{72}}\|\theta_0\|_{L^p(\T^2)},\qquad &\forall p \in [1,\infty],\\
        \label{eq:cor-Hs}
        \|\theta^\kappa(1,\cdot)\|_{H^\sigma(\T^2)} &\leq C\kappa^{\frac{(1-\alpha)^2}{144}}\|\theta_0\|_{H^{-\sigma}(\T^2)},\qquad &\forall \sigma \in \Big[0,\tfrac{(1-\alpha)^2}{288}\Big].
    \end{align}
\end{corollary}

We note the following consequence of Theorem~\ref{thm:total_dissipation} on stochastic differential equations. If $X_t^\kappa$ is a solution to
\begin{equation}
\label{eq:stochastic-intro}\begin{cases}
dX_t^\kappa=V(t,X_t)\,dt+\sqrt{2\kappa}\,dw_t,\\
X_0^\kappa\sim\theta_0(dx),
\end{cases}\end{equation}
where $w_t$ is a standard Brownian motion on $\R^2$
and $\theta_0$ is an arbitrary probability measure on $\T^2$, then $\mathcal{L}(X_t^\kappa)$---the law of $X^\kappa_t$---is equal to the solution to~\eqref{eq:intro_drift_diffusion_equation}, $\theta^{\kappa}(t,x)\,dx$. Thus by Theorem~\ref{thm:total_dissipation},
\begin{equation}
\label{eq:spontaneous-stochasticity}
\Big\|\mathcal{L}(X_1^\kappa)-\big|\T^2\big|^{-1}dx\Big\|_{TV(\T^2)}\leq C\kappa^{(1-\alpha)^2/72}.
\end{equation}
That is, $X^\kappa_1$ converges to the uniform measure on $\T^2$ in total variation uniformly over all initial laws. The behavior of the stochastic trajectories solving~\eqref{eq:stochastic-intro} is further discussed in Subsection~\ref{subsec:spread}.

\subsection{Background}

The anomalous dissipation of energy is a foundational phenomenon in the study of fluid and passive scalar turbulence. Originating in fluid turbulence, anomalous dissipation describes the tendency for a turbulent fluid to dissipate energy at a constant rate independent of the molecular viscosity, despite viscosity being the ultimate mechanism of dissipation. See~\cite[Section 5.2]{frisch_turbulence_1995} for a lucid discussion of the phenomenon and its empirical evidence.

Anomalous dissipation is a cornerstone of phenomenological turbulence theory and is sometimes called the \textit{zeroth law of turbulence}. It is taken as a basic axiom to Kolmogorov's highly successful K41 theory~\cite{kolmogorov_degeneration_1941,kolmogorov_dissipation_1941,kolmogorov_local_1941}. A rigorous justification of anomalous dissipation in fluids derived purely from the incompressible fluid equations remains a major goal of mathematical fluid research. This goal however is well out of reach of current techniques. Instead, research has focused on the passive scalar problem, in which we study solutions $\theta^\kappa$---called \textit{passive scalars}---to the linear equation~\eqref{eq:intro_drift_diffusion_equation} where the advecting flow $V$ is specified independently (as opposed to being dynamically determined by a PDE, as in the case of fluid dynamics). When $V$ is the velocity field of a real fluid, the passive scalar equation describes the evolution of tracers---such as smoke or dye concentrations. 

The dynamics of passive scalars under suitably ``fluid-like'' advecting flows is called \textit{passive scalar turbulence} as the passive scalars exhibit many of the same phenomena as turbulent fluids. In particular, anomalous dissipation is expected for a broad variety of ``turbulent'' advecting flows and is an essential input to the heuristic phenomenological theory mirroring K41 theory developed by Obukhov~\cite{obukhov_structure_1949} and Corrsin~\cite{corrsin_spectrum_1951}. 

To be more precise, for solutions $\theta^\kappa$ to the passive scalar equation~\eqref{eq:intro_drift_diffusion_equation} and $\kappa>0$, using that $\nabla \cdot V=0,$ we have the fundamental energy identity
\begin{equation}
\label{eq:energy-identity}
\frac{d}{dt} \frac{1}{2} \|\theta^\kappa\|_{L^2}^2 = - \kappa \|\nabla \theta^\kappa\|_{L^2}^2.
\end{equation}
Thus, we might naively expect that as $\kappa \searrow 0$, we get asymptotic energy conservation
\[\lim_{\kappa \to 0} \|\theta^{\kappa}(t,\cdot)\|_{L^2} = \|\theta(0,\cdot)\|_{L^2}.\]
If the advecting flow $V$ is sufficiently smooth (e.g.\ $L^1_t W^{1,\infty}_x$), then one can bound $\|\nabla \theta^\kappa\|_{L^2}$ uniformly in $\kappa$ and prove that the above asymptotic energy conservation holds. However, for less regular flows (e.g.\ $L^\infty_t C^\alpha_x$), it is possible that $\|\nabla \theta^\kappa\|_{L^2}^2 \approx \kappa^{-1}$, and so 
\[\limsup_{\kappa \to 0} \|\theta^\kappa(t,\cdot)\|_{L^2} < \|\theta(0,\cdot)\|_{L^2}.\]
It is this phenomenon---the failure of energy to be conserved in the vanishing diffusivity limit---that we call the \textit{anomalous dissipation} of energy for passive scalars.

\subsection{Previous work}

\subsubsection{Anomalous dissipation examples}

Despite passive scalars being dramatically easier to study than fluids, examples of advecting flows that exhibit anomalous dissipation have only been given quite recently. Starting with~\cite{drivas_anomalous_2022}, there has been a profusion of examples:~\cite{colombo_anomalous_2023,armstrong_anomalous_2025,burczak_anomalous_2023,elgindi_norm_2024,rowan_anomalous_2024,hofmanova_anomalous_2023,rowan_accelerated_2024,johansson_anomalous_2024} among others. 

Let us briefly describe some aspects of these contributions. The velocity fields of \cite{drivas_anomalous_2022,colombo_anomalous_2023,elgindi_norm_2024} are essentially built on mixing examples. They prove anomalous dissipation by carefully understanding mixing dynamics---which transports Fourier mass from small to large wavenumbers---and treating the diffusion as a perturbation. The strongest such result is given by~\cite{elgindi_norm_2024}, in which they provide a vector field $u \in C^\infty([0,1], C^{1^-}(\T^2))$ that anomalously diffuses all smooth, mean-zero initial data as $\kappa \to 0$.

In \cite{johansson_anomalous_2024}, the authors take a fairly different perspective in order to build an autonomous 2D vector field $V$ that exhibits anomalous dissipation. Their proof works by carefully tracking stochastic characteristics and utilizing the connection between anomalous dissipation and spontaneous stochasticity, discussed below.

In~\cite{rowan_anomalous_2024}, the second author proves anomalous dissipation for the Kraichnan model, a stochastic advecting flow that is given by a white-in-time correlated-in-space Gaussian field which was introduced in the physics literature~\cite{kraichnanSmallScaleStructure1968} and given extensive study in the applied and physics literature~\cite{gawedzki_anomalous_1995,bernard_slow_1998,falkovichParticlesFieldsFluid2001}.

In~\cite{hofmanova_anomalous_2023}, the authors apply a singular-in-time stochastic forcing to the Navier--Stokes equations to generate a singular-in-time advecting flow that totally dissipates all mean-zero initial data: $\|\theta^\kappa(1,\cdot)\|_{L^2} = 0$ for all $\kappa>0.$ In~\cite{rowan_accelerated_2024} the second author shows all relaxation enhancing flows (\cite{constantin_diffusion_2008}) generate the same total dissipation phenomenon if suitably ``accelerated''. In both~\cite{hofmanova_anomalous_2023,rowan_accelerated_2024}, the advecting flow $V$ fails to be in $L^1([0,1], L^\infty(\T^2)).$

In~\cite{armstrong_anomalous_2025}, the authors use fractal homogenization---a very different technique than those used above---to prove anomalous dissipation. Their ``fluid-like'' construction---together with the proof techniques---was shown in~\cite{burczak_anomalous_2023} to be robust enough so that the advecting flow could be perturbed using convex integration to be a weak solution to the Euler equation while preserving anomalous dissipation. The advecting flow solving a fluid equation is highly desirable for treating realistic passive scalar turbulence. In~\cite{armstrong_anomalous_2025,burczak_anomalous_2023}, the anomalous dissipation is also present for all smooth mean-zero initial data, though only along a subsequence $\kappa_j \to 0.$ 

\subsubsection{Connections with other problems}

In~\cite{drivasLagrangianFluctuationDissipation2017}, the authors prove a quantitative connection between anomalous dissipation for scalars and \textit{spontaneous stochasticity}. Spontaneous stochasticity refers to non-trivial variance in the distribution of a particle being advected by the flow and forced by a white noise under the limit of vanishing noise intensity. A version of this phenomenon is evident in~\eqref{eq:spontaneous-stochasticity} as well as Subsection~\ref{subsec:spread}. This connection provides an alternative route for the study of anomalous dissipation and is exploited in~\cite{johansson_anomalous_2024}.

Additionally, as is noted in~\cite[Theorem 1.3]{rowan_anomalous_2024} but was essentially implicit in~\cite{drivasLagrangianFluctuationDissipation2017,drivas_anomalous_2022}, anomalous dissipation implies non-uniqueness of solutions to the transport PDE given by~\eqref{eq:intro_drift_diffusion_equation} with $\kappa=0$ as well as non-uniqueness of ODE solutions to $\dot x_t = V(t,x_t)$.

Lastly, the problem of anomalous dissipation of passive scalars is strongly related to the problem of anomalous dissipation for forced fluid equations~\cite{brue_anomalous_2023,cheskidov_dissipation_2024} through the $2+ \frac{1}{2}$ dimensional construction. In particular~\cite{cheskidov_dissipation_2024} uses the perfect mixing example of~\cite{alberti_exponential_2019} to generate total dissipation in this forced fluid equation setting. The \cite{alberti_exponential_2019} flow will also serve as the basis of our construction.

\subsection{Contributions}

\label{subsec:contribs}

The strongest result given by previous works is anomalous dissipation of the form
\[\limsup_{\kappa \to 0} \|\theta^\kappa(1,\cdot)\|_{L^2} \leq \rho(\theta_0) \|\theta_0\|_{L^2},\]
where the rate $\rho(\theta_0) < 1$ depends on $\|\nabla \theta\|_{L^2}\|\theta\|_{L^2}^{-1}$. In particular, the rate \textit{is not uniform} for all mean-zero $\theta_0\in L^2(\T^2)$. As such, our primary contribution is the first example of an anomalously dissipating flow that dissipates with a uniform rate for all mean-zero data in $L^2(\T^2)$. Further, as is shown in Theorem~\ref{thm:total_dissipation}, we actually treat $L^1$ (even $TV$) initial data, instead of just $L^2$ initial data, and still get a uniform rate of dissipation.

Additionally, we prove \textit{asymptotic total dissipation}, that is 
\[\limsup_{\kappa \to 0} \|\theta^\kappa(1,\cdot)\|_{L^2} = 0.\]
Total dissipation was studied in the forced fluid equation setting in~\cite{cheskidov_dissipation_2024}. However, total dissipation has not yet been studied in the passive scalar anomalous dissipation setting nor has it ever been shown for all initial data. 

Lastly, we provide quantitative bounds on the dissipation for all $\kappa>0.$ Namely, we get an explicit algebraic rate of the vanishing of the energy as $\kappa \to 0$, as is seen in~\eqref{eq:main-theorem-estimate}. This algebraic rate allows us to combine the dissipation estimate with the smoothing of the heat kernel to get uniform total dissipation of negative regularity initial data as well as to get the vanishing of positive regularity norms at the final time (as $\kappa \to0$), as is stated in~\eqref{eq:cor-Hs}. 

These results are all shown for velocity fields with good regularity: $V \in L^\infty_tC^\alpha_x \cap C^{\frac{\alpha}{1-\alpha}}_t L^\infty_x$. A technical refinement of the construction would allow the velocity field to take $\alpha= 1^-$, that is one could construct a single velocity field $V \in C^\infty([0,1], C^\alpha(\T^2))$ for all $\alpha \in (0,1)$ as in~\cite{elgindi_norm_2024}. Doing this would give a sub-algebraic rate in $\kappa$ of total dissipation however and substantially complicates the exposition.

The velocity field constructed here is very special and as such fairly ``unphysical'': in particular, it is neither ``generic'' nor ``fluid-like''. Additionally, the solutions $\theta^\kappa$ do not obey the Obukhov--Corrsin scaling studied in~\cite{drivas_anomalous_2022,colombo_anomalous_2023,elgindi_norm_2024}. In particular, since the limiting solution to the transport equation as $\kappa \to 0$ is spatially discontinuous---as is made clear below---we cannot possibly have uniform-in-$\kappa$ bounds on $\theta^\kappa$ in some H\"older space. Further, the dissipation measure---describing the time distribution of the energy dissipation in the $\kappa \to 0$ limit---is purely atomic, as opposed to the non-atomic dissipation measures of~\cite{armstrong_anomalous_2025,burczak_anomalous_2023,johansson_nontrivial_2023,johansson_anomalous_2024}. Physically the dissipation should be at least non-atomic. For a precise characterization of the dissipation measure for the flow $V$ that we construct, see Remark~\ref{rem:dissipation_measure}. Finally, asymptotic total dissipation is not even expected for generic fluid flows, so the main result is arguably a point against its physicalness. 

However, in some sense the velocity fields we construct exhibit the \textit{strongest possible passive scalar anomalous dissipation.} Due to the unique continuation result of~\cite{poon_unique_1996}, we cannot have total dissipation when $\kappa>0$: if $\theta_0 \ne 0,$ then $\|\theta^\kappa(1,\cdot)\|_{L^2}>0$. As such, it is not clear what a stronger anomalous dissipation result would be.

\subsection{Outline of contents}

In Section~\ref{sec:definition-outline}, we define the flow $V$ for which we prove Theorem~\ref{thm:total_dissipation}, as well as an intermediate flow $v$ used to construct $V$. Additionally, we provide a broad heuristic description of the argument followed by a much more precise outline of the proof. The remainder of the paper then serves to fill in the missing technical details of the argument. In Section~\ref{sec:two-cell}, we prove the special case of Theorem~\ref{thm:total_dissipation} given by Theorem~\ref{thm:two-cell-dissipation} and in Section~\ref{sec:universal}, we prove Theorem~\ref{thm:total_dissipation}. In Appendix~\ref{appendix:stability}, we prove all the stability estimates on drift-diffusion equations used as lemmas in Sections~\ref{sec:two-cell} and~\ref{sec:universal}. The arguments there rely on the stochastic representation of the drift-diffusion equation. Finally, in Appendix~\ref{appendix:regularity}, we provide the computational verification of the regularity of the constructed flows.

\subsection{Acknowledgments}
We would like to thank Theodore Drivas for pointing out the maximal spreading of the stochastic trajectories noted in Subsection~\ref{subsec:spread}. 

EHC was partially supported by NSF grant DMS-2342349. KR was partially supported by the NSF Collaborative Research Grant DMS-2307681 and the Simons Foundation through the Simons Investigators program. This paper was written in part during KR's visit to Carnegie Mellon University, made possible by the Frontiers in Applied Analysis RTG funded by NSF grant DMS-2342349.

\section{Definition of flows and proof outline}

\label{sec:definition-outline}

Let us first provide a highly heuristic description of the argument. Then we will go through the argument again, explaining how to deal with the technical problems that arise in making the heuristic precise. In the process, we give the precise definitions of the flows under study. We additionally note---though only provide the heuristic argument---an interesting property of the vector field $V$ in Subsection~\ref{subsec:spread}. We also provide the proof of Corollary~\ref{cor:other-dissipations} at the end of this section.

\subsection{Heuristic argument}

The argument proceeds in two major steps. The first is building a \textit{two-cell dissipator} which asymptotically totally dissipates special initial data: the \textit{two-cell data}. The two-cell data is given on the box $B:= [0,\sqrt{2}]\times [0,1]$ by $\indc_{\{x <\sqrt{2}/2\}}$. The action of the two-cell dissipator, in the $\kappa \to 0$ limit, is to then take that two-cell data at time $0$ to the constant function $\tfrac{1}{2}$ at time $1$, as illustrated in Figure~\ref{fig:two-cell}.

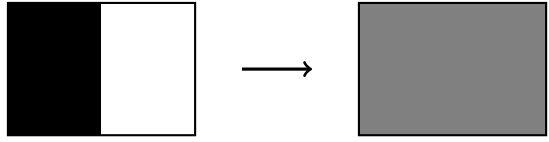
\begin{figure}[htbp]
    \centering
\begin{tikzpicture}

\def\width{0.88*0.3535}  
\def\height{0.88*0.5}        

\fill[black!100!white] (33*\width, 0) rectangle (37*\width,4*\height);

\fill[black!0!white] (37*\width, 0) rectangle (41*\width,4*\height);

\draw[thick] (33*\width, 0) rectangle (41*\width, 4*\height);

\draw[->,very thick,black] (43*\width,2*\height) -- (46*\width,2*\height);


\fill[black!50!white] (48*\width, 0) rectangle (56*\width,4*\height);

\draw[thick] (48*\width, 0) rectangle (56*\width, 4*\height);

\end{tikzpicture}
    \caption{The two-cell dissipator.}
    \label{fig:two-cell}
\end{figure}

The construction of the two-cell dissipator is based on the perfect mixing example of~\cite{alberti_exponential_2019}. Their example is readily modified to construct a divergence-free flow that is $C^\infty_{loc,t,x}([0,1) \times B) \cap L^\infty([0,1], C^\alpha(B))$ that---when $\kappa=0$---takes the two-cell data at time $0$ to the constant function at time $1$, where we are using that the flow is smooth away from the final time, so the transport equation is well defined. Proving that this perfect mixing example is robust under perturbation by $\kappa \Delta$ so as to asymptotically totally dissipate requires carefully analyzing the almost self-similar and geometric structure of the flow. We defer this analysis for now.

With the two-cell dissipator in hand, we now discuss how we build the \textit{universal total dissipator}. The central observation is that if our initial data was piecewise constant on dyadic boxes of scale $2^{-n}$, then by applying the two-cell dissipator---rescaled and tiled---we can flow the data under the equation to be essentially piecewise constant on boxes on the larger scale $2^{-n+1}$. That is \textit{using the two-cell dissipator, we can essentially take piecewise constant data on a small scale to a piecewise constant function on a larger scale.} We can then iterate this construction, taking data that is piecewise constant on scale $2^{-n} \to 2^{-n+1} \to \cdots \to 1$, as is shown in Figure~\ref{fig:total-dissipator}. Additionally, by taking the space-time rescaling to preserve the $L^\infty_tC^\alpha_x$ norm of the flow, we can make each smaller stage of the flow take geometrically less time. As such, the flow illustrated by Figure~\ref{fig:total-dissipator} takes unit time, independent of how large $n$ is.

\begin{figure}[htbp]
    \centering
\begin{tikzpicture}

\def\width{0.88*0.3535}  
\def\height{0.88*0.5}  


\draw[white,fill=white] (0,2*\height) circle (0.01);

\draw[black,fill=black] (3*\width,2*\height) circle (0.05);

\draw[black,fill=black] (4*\width,2*\height) circle (0.05);

\draw[black,fill=black] (5*\width,2*\height) circle (0.05);

\draw[->,very thick,black] (9*\width,2*\height) -- (10*\width,2*\height);


\fill[black!25!white] (11*\width, 0) rectangle (13*\width, 2*\height);

\fill[black!45!white] (13*\width, 0) rectangle (15*\width, 2*\height);

\fill[black!100!white] (11*\width, 2*\height) rectangle (13*\width, 4*\height);

\fill[black!70!white] (13*\width, 2*\height) rectangle (15*\width, 4*\height);

\fill[black!20!white] (15*\width, 0) rectangle (17*\width, 2*\height);

\fill[black!10!white] (17*\width, 0) rectangle (19*\width, 2*\height);

\fill[black!10!white] (15*\width, 2*\height) rectangle (17*\width, 4*\height);

\fill[black!0!white] (17*\width, 2*\height) rectangle (19*\width, 4*\height);

\draw[thick] (11*\width, 0) rectangle (19*\width, 4*\height);

\draw[<->,thick,red] (12*\width,1*\height) -- (14*\width,1*\height);

\draw[<->,thick,red] (16*\width,1*\height) -- (18*\width,1*\height);

\draw[<->,thick,red] (12*\width,3*\height) -- (14*\width,3*\height);

\draw[<->,thick,red] (16*\width,3*\height) -- (18*\width,3*\height);

\draw[->,very thick,black] (20*\width,2*\height) -- (21*\width,2*\height);


\fill[black!35!white] (22*\width, 0) rectangle (26*\width, 2*\height);

\fill[black!85!white] (22*\width, 2*\height) rectangle (26*\width, 4*\height);

\fill[black!15!white] (26*\width,0) rectangle (30*\width, 2*\height);

\fill[black!5!white] (26*\width,2*\height) rectangle (30*\width, 4*\height);

\draw[thick] (22*\width, 0) rectangle (30*\width, 4*\height);

\draw[<->,very thick,red] (24*\width,1*\height) -- (24*\width,3*\height);

\draw[<->,very thick,red] (28*\width,1*\height) -- (28*\width,3*\height);

\draw[->,very thick,black] (31*\width,2*\height) -- (32*\width,2*\height);


\fill[black!60!white] (33*\width, 0) rectangle (37*\width,4*\height);

\fill[black!10!white] (37*\width, 0) rectangle (41*\width,4*\height);

\draw[thick] (33*\width, 0) rectangle (41*\width, 4*\height);

\draw[<->,ultra thick,red] (35*\width, 2*\height) -- (39*\width,2*\height);

\draw[->,very thick,black] (42*\width,2*\height) -- (43*\width,2*\height);


\fill[black!35!white] (44*\width, 0) rectangle (52*\width,4*\height);

\draw[thick] (44*\width, 0) rectangle (52*\width, 4*\height);

\end{tikzpicture}
    \caption{The universal total dissipator. The red arrows denote the action of the two-cell dissipator.}
    \label{fig:total-dissipator}
\end{figure}
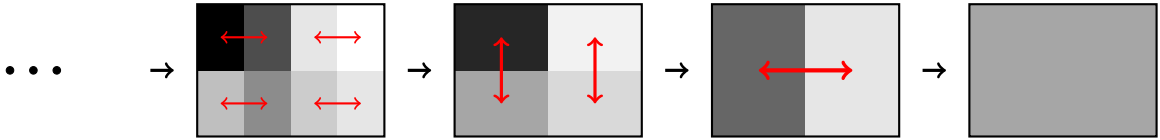

By then taking $n \to \infty$, we can take data that is piecewise constant on scale $2^{-\infty} =0$ to a constant function on scale $1$, all in the unit time interval. Since every function is piecewise constant on scale $0$, this says we can totally dissipate any data.

\begin{remark}[Dissipation measure]
\label{rem:dissipation_measure}

The dissipation measure is the (weak) limit as $\kappa \searrow 0$ of the energy dissipation for the drift-diffusion equation~\eqref{eq:intro_drift_diffusion_equation}, which utilizing~\eqref{eq:energy-identity} is given as
\[\lim_{\kappa \to 0} \kappa \int |\nabla \theta^\kappa|^2(t,x)\,dx =: \mathscr{E}(dt),\]
where the limiting object $\mathscr{E}(dt)$ is generically only a measure. In~\cite{johansson_nontrivial_2023,johansson_anomalous_2024} they construct anomalously dissipating flows for which the (subsequential) limiting object $\mathscr{E}(dt) \in L^\infty([0,1])$ and in~\cite{armstrong_anomalous_2025,burczak_anomalous_2023} they have $\mathscr{E}(dt)$ is nonatomic, but potentially not absolutely continuous with respect to Lebesgue measure. In the examples built on mixing flows, such as~\cite{drivas_anomalous_2022,colombo_anomalous_2023,elgindi_norm_2024}, $\mathscr{E}(dt) = c_{\theta_0} \delta_{1}(dt)$. That is, all dissipation happens at the final time (in the $\kappa \to 0$ limit). Since our two-cell dissipator is built on a mixing example, it will be clear from the proof below that $\mathscr{E}^{\text{two-cell}}(dt) = \delta_1(dt)$ for the two-cell initial data. Then, as described above, the universal total dissipator is built by iterating the two-cell dissipating example. As such, one can readily see from the proof that for the universal total dissipator
\[\mathscr{E}(dt) = \sum_j c^j_{\theta_0}\delta_{t_j}(dt),\]
where $t_j$ are a sequence of geometrically separated times decreasing to $\tfrac{1}{2}$---see Definition~\ref{def:V}---and the $c^j_{\theta_0}$ are explicitly computable constants that give the difference of the $L^2$ energy of $\theta_0$ averaged on grid cells of scale $2^{-(j+1)/2}$ and $\theta_0$ averaged on grid cells of scale $2^{-j/2}$.

Therefore, the dissipation measure is purely atomic, as stated in Subsection~\ref{subsec:contribs}. It is however somewhat more diffuse than the dissipation measures for anomalously dissipating flows built purely on mixing examples, though it is nowhere near as diffuse as the dissipation measure given in e.g.~\cite{johansson_anomalous_2024}. This is nonphysical as we expect the dissipation of generic turbulent fluids to be continuous in time.
\end{remark}

\subsection{Constructing the two-cell dissipator}

\label{subsec:two-cell-sketch}

Let us remark on the conventions we will use throughout the paper. We choose to always work on the box $[0,\sqrt{2}] \times [0,1]$ as it is self-similar under being split in half, that is it has two halves that are congruent to the original box. This somewhat simplifies the construction of the universal total dissipator.

\begin{definition}\label{def:box}
We let $B:=[0,\sqrt{2}] \times [0,1]$ and take $\T^2$ to be $B$ with opposite sides identified.  
\end{definition}

We now introduce notation for the solution operator to the drift-diffusion equation.

\begin{definition}
\label{def:sol-operator}
    Given a vector field $u\in L^\infty([0,\infty)\times B)$, boundary data $f \in L^\infty([0,\infty) \times \partial B)$ and $\kappa> 0$, for any $0 \leq s \leq t$, $\sol_{s,t}^{u,\kappa,f} : TV(B) \to L^1(B)$ denotes the solution operator to
    \begin{equation}\label{eq:drift_diff_equation}
    \begin{cases}
    \partial_t \theta - \kappa \Delta \theta + u \cdot \nabla \theta=0&[s,t]\times B,
    \\ \theta(\cdot,x) = f(\cdot,x)&x\in \partial B.
    \end{cases}
    \end{equation}
    For $\kappa=0$, we use $\sol^{u,0,f}_{s,t}$ to denote the solution to~\eqref{eq:drift_diff_equation}, though in this case we require $u \in L^\infty([0,\infty), W^{1,\infty}(B))$.
    If $u$ is tangent to $\partial B$, then $\sol_{s,t}^u:=\sol_{s,t}^{u,0,f}$ for all boundary data $f$. We denote $\sol^{u,\kappa,\T^2}_{s,t}$ the solution operator to the problem with periodic boundary data.
\end{definition}

Our goal for the two-cell dissipator is to prove the following result.

\begin{theorem}
\label{thm:two-cell-dissipation}
For $\alpha \in(0,1)$, let $v \in L^\infty([0,1], C^\alpha(B)) \cap C^{\frac{\alpha}{1-\alpha}}([0,1], L^\infty(B))$ be defined below by~\eqref{eq:v-def}. Then there exists $C(\alpha)>0$ such that for all boundary data $f\in L^\infty([0,1]\times\partial B)$ and $\kappa>0$, we have the estimate
\[\big\|\sol_{0,1}^{v,\kappa,f}\Theta_0-\tfrac{1}{2}\big\|_{L^1(B)}\leq C\big(1 +\|f\|_{L^\infty([0,1]\times\partial B)}\big)\kappa^{(1-\alpha)/12},\]
where $\Theta_0(x,y) := \indc_{\{x < \frac{\sqrt{2}}{2}\}}.$
\end{theorem}

We make the following observations about Theorem~\ref{thm:two-cell-dissipation} that may be helpful to keep in mind. 

\begin{itemize}
    \item We are constructing a vector field for each $\alpha \in (0,1)$ that is spatially $C^\alpha$. The spatial regularity is by construction, but the time regularity \text{comes from} a computation. It is interesting to note that the time regularity $\frac{\alpha}{1-\alpha}\to \infty$ as $\alpha \to 1.$
    \item For each $\alpha \in (0,1)$, we are getting an algebraic rate in $\kappa$ that degenerates to $0$ as $\alpha \to 1.$ This is sensible as we cannot have anomalous dissipation with a spatially Lipschitz vector field.
    \item We are solving the advection-diffusion equation as a boundary value problem on the box---as opposed to the problem on the torus which lacks spatial boundary. This is because when we tile the two-cell dissipators to make the universal total dissipator, we have to deal with the lack of spatial periodicity on small scales. That is, we have to deal with the interaction across different ``tiles''. We see from the statement of Theorem~\ref{thm:two-cell-dissipation} that we are treating the boundary data essentially as pure error.
\end{itemize}

\subsubsection{\texorpdfstring{The~\cite{alberti_exponential_2019} perfect mixing flow}{The [ACM19] perfect mixing flow}}

The starting place for our perfect mixing flow is a self-similar mixer. \textit{Note however that this is not the flow we use, as one cannot get self-similar mixing with a Lipschitz vector field.} A self-similar mixing flow takes some initial data at time $0$---in Figure~\ref{fig:self-similar} an indicator on a circle---and periodizes the data at time $1$, flowing under the transport equation $(\kappa =0)$. The flow is then periodized in space---which preserves Lipschitz norms and contracts $L^\infty$ norms by a factor of $2$---on the time interval $[1,2]$. The action of this flow on the solution is then to again periodize the solution: from being $2^{-1}$ periodic to being $2^{-2}$ periodic. Iterating this construction, we get a flow that has constant spatial Lipschitz norm, has $L^\infty_x$ norm $\approx 2^{-n}$ on the time interval $[n,n+1]$, and takes a specific initial data (indicator on a circle) to the $2^{-n}$ periodized version of the data at time $n$. Note that at time $n$, on each grid cell of scale $2^{-n}$, exactly half of the cell is colored black and half is colored white.

\begin{figure}[htbp]
    \centering
\begin{tikzpicture}

\def\width{2}  

\draw[thick] (0,0) rectangle (\width,\width);
\fill[black!100!white] (\width/2, \width/2) circle (\width/2.50662827463);
\draw[->, very thick,black] (1.33*\width, \width/2) -- (1.66*\width, \width/2);
\draw[thick] (2*\width,0) rectangle (3*\width,\width);
\fill[black!100!white] (\width/4+2*\width, \width/4) circle (0.5*0.3989422804*\width);
\fill[black!100!white] (3*\width/4+2*\width, \width/4) circle (0.5*0.3989422804*\width);
\fill[black!100!white] (\width/4+2*\width, 3*\width/4) circle (0.5*0.3989422804*\width);
\fill[black!100!white] (3*\width/4+2*\width, 3*\width/4) circle (0.5*0.3989422804*\width);

\draw[->, very thick,black] (3.33*\width, \width/2) -- (3.66*\width, \width/2);

\draw[thick] (4*\width,0) rectangle (5*\width,\width);

\fill[black!100!white] (\width/8+4*\width, \width/8) circle (0.5*0.5*0.3989422804*\width);
\fill[black!100!white] (3*\width/8+4*\width, \width/8) circle (0.5*0.5*0.3989422804*\width);
\fill[black!100!white] (5*\width/8+4*\width, \width/8) circle (0.5*0.5*0.3989422804*\width);
\fill[black!100!white] (7*\width/8+4*\width, \width/8) circle (0.5*0.5*0.3989422804*\width);

\fill[black!100!white] (\width/8+4*\width, 3*\width/8) circle (0.5*0.5*0.3989422804*\width);
\fill[black!100!white] (3*\width/8+4*\width, 3*\width/8) circle (0.5*0.5*0.3989422804*\width);
\fill[black!100!white] (5*\width/8+4*\width, 3*\width/8) circle (0.5*0.5*0.3989422804*\width);
\fill[black!100!white] (7*\width/8+4*\width, 3*\width/8) circle (0.5*0.5*0.3989422804*\width);

\fill[black!100!white] (\width/8+4*\width, 5*\width/8) circle (0.5*0.5*0.3989422804*\width);
\fill[black!100!white] (3*\width/8+4*\width, 5*\width/8) circle (0.5*0.5*0.3989422804*\width);
\fill[black!100!white] (5*\width/8+4*\width, 5*\width/8) circle (0.5*0.5*0.3989422804*\width);
\fill[black!100!white] (7*\width/8+4*\width, 5*\width/8) circle (0.5*0.5*0.3989422804*\width);

\fill[black!100!white] (\width/8+4*\width, 7*\width/8) circle (0.5*0.5*0.3989422804*\width);
\fill[black!100!white] (3*\width/8+4*\width, 7*\width/8) circle (0.5*0.5*0.3989422804*\width);
\fill[black!100!white] (5*\width/8+4*\width, 7*\width/8) circle (0.5*0.5*0.3989422804*\width);
\fill[black!100!white] (7*\width/8+4*\width, 7*\width/8) circle (0.5*0.5*0.3989422804*\width);

\draw[->, very thick,black] (5.33*\width, \width/2) -- (5.66*\width, \width/2);

\fill[black!100!white] (6.33*\width-0.12*\width, \width/2) circle (0.03*\width);
\fill[black!100!white] (6.33*\width, \width/2) circle (0.03*\width);
\fill[black!100!white] (6.33*\width+0.12*\width, \width/2) circle (0.03*\width);







\end{tikzpicture}
    \caption{A self-similar mixing flow. The arrows represent the flow of the solution under the transport equation for unit time.}
    \label{fig:self-similar}
\end{figure}
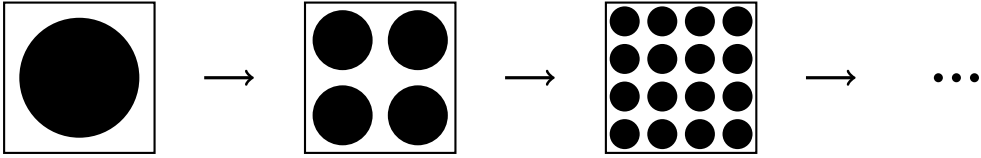

In~\cite[Section 5]{alberti_exponential_2019}, they construct a flow that does exactly what is described above. However, this flow fails to be spatially Lipschitz, as is apparent due to the topological changes in the level sets illustrated in Figure~\ref{fig:self-similar}. We want the flow to have good spatial regularity in order to show that the effect of the diffusion is appropriately small. We thus instead use the construction of~\cite[Section 8]{alberti_exponential_2019}. In particular, we want to use the below result, which is readily extracted from their construction. First, we introduce the following notation to decompose the box $B$ so that $B$ is the (essentially) disjoint union of the scale $5^{-n}$ sub-boxes $(B_n + x_n)_{x_n \in \Pi_n}$.

\begin{definition}
    Let 
    \begin{align*}
    B_n&:=5^{-n}B,\\
    \Pi_n &:=\Big\{5^{-n}(\sqrt{2} k,\ell):0\leq k,\ell\leq 5^n-1\Big\}.
    \end{align*}
\end{definition}

\begin{theorem}[\cite{alberti_exponential_2019}]
\label{thm:ACM_field}
There exists a divergence free vector field $U\in C^\infty\big([0,\infty),W^{1,\infty}(B)\big)$ tangential to $\partial B$ and subsets $E_n\subseteq B$ for $n\geq 0$ so that the following properties hold.
\begin{enumerate}
\item\label{item:dies} $\partial_t^k U(n,\cdot) =0$ for all $k,n\geq 0,$ 
\item\label{item:periodic} $U$ extends to a continuous periodic function on $\T^2$,
\item\label{item:box_averages} $|E_n\cap(B_n+x_n)| = \tfrac{1}{2}|B_n| =\frac{\sqrt{2}}{2}5^{-2n}$ for all $n\geq 0$ and $x_n\in \Pi_n$,
\item\label{item:boundary_curves} There exists $C>0$ so that $\partial E_n$ can be parametrized by a Lipschitz curve with length less than $C5^j$ for all $n\geq0$,
\item\label{item:set_mapping}Letting $\Theta_n:=\indc_{E_n}$, it holds that $\sol^U_{0,n}\Theta_0=\Theta_n,$
\item\label{item:L^infty_bound} $\|U\|_{C^k([n,n+1],L^\infty(B))}\leq C_k5^{-n}$ for all $n,k \geq 0$,
\item $E_0=\Big\{(x,y) \in B:x<\frac{\sqrt{2}}{2}\Big\}.$ \label{item:E0}
\end{enumerate}
\end{theorem}

The flow given by Theorem~\ref{thm:ACM_field} takes the two-cell initial data $\Theta_0$ at time $0$ and flows it under the transport equation to be $\Theta_n := \indc_{E_n}$ at time $n$. We want that $E_n$ ``looks like'' a $5^{-n}$ periodization of the initial data. Because we want the flow to be spatially Lipschitz, this however cannot be exactly true, so~\cite{alberti_exponential_2019} use instead a ``quasi-self-similar'' construction. The important ``quasi-self-similar'' properties of the $E_n$ are contained in Items~\ref{item:box_averages} and~\ref{item:boundary_curves}. Item~\ref{item:box_averages} gives that $\Theta_n$ has the same average on every grid cell of size $5^{-n}$. Item~\ref{item:boundary_curves} is essentially a regularity statement on $\Theta_n$, which approximately gives that $E_n$ has the same regularity as a $5^{-n}$ periodization of the initial data. The important ``quasi-self-similar'' properties of $U$ are given by $\sup_t \|U(t,\cdot)\|_{W^{1,\infty}} = C <\infty$ and by Item~\ref{item:L^infty_bound}, which says that pointwise in space $U$ is order $5^{-n}$ on the time interval $[n,n+1]$. Finally, Items~\ref{item:dies} and~\ref{item:periodic} are just there for verifying regularity properties of the vector fields we construct. We recommend the reader view the figures of~\cite[Section 8]{alberti_exponential_2019} to get a geometric picture of their construction.

Let us briefly comment on how Theorem~\ref{thm:ACM_field} is implied by~\cite[Section 8]{alberti_exponential_2019}. Every item except Item~\ref{item:E0} is effectively direct from their construction and argument. One minor technical difference to note is that~\cite{alberti_exponential_2019} is tracking the flow of tubes under the transport equation while we consider the indicator of the half-plane $\Theta_0 := \indc_{E_0}$. It is easy to see however that this makes no difference, as in both cases the evolution is determined purely by the evolution of the boundary of the half-plane, denoted $\Gamma_1(t), \Gamma_2(t)$ in~\cite{alberti_exponential_2019}. Finally, there is the issue that the initial configuration in~\cite{alberti_exponential_2019} is not given by $\Theta_0$. It is however easy to add an additional flow to transform $\Theta_0$ into the initial configuration of~\cite{alberti_exponential_2019}. This flow is outlined precisely in~\cite[Figure 11(a)]{alberti_exponential_2019}.

\subsubsection{The two-cell dissipating flow}

We now rescale the vector field of~\cite{alberti_exponential_2019} as given by Theorem~\ref{thm:ACM_field} in time in such a way as to compress $[0,\infty)$ down to the finite interval $[0,1/2)$ while maintaining H\"older regularity. This then gives us the vector field with which we will prove Theorem~\ref{thm:two-cell-dissipation}. 
\begin{definition} 
 Define $\tau_j:=\frac{1}{2}\big(5^{(\alpha-1)j}-5^{(\alpha-1)j+1}\big)$ and $t_n:= \sum_{j=0}^{n-1} \tau_j.$ Then letting $U$ be as in Theorem~\ref{thm:ACM_field}, 
we define the divergence-free vector field $v : [0,1] \times B \to \R^2$ by
\begin{equation}
\label{eq:v-def}
v(t,x):=\begin{cases}
    \tau_n^{-1} U(\tau_n^{-1}(t-t_n)+n,x)& t\in[t_n,t_{n+1}],
\\0&t\geq 1/2.
\end{cases}
\end{equation}
\end{definition} 

We defer the computational verification of the following regularity of $v$ to Appendix~\ref{appendix:regularity}. 

\begin{lemma}
    \label{lem:v-props}
    The vector field $v$ defined by~\eqref{eq:v-def} has the following properties:
    \begin{enumerate}
        \item \label{item:v-regular} $v \in L^\infty([0,1], C^\alpha(B)) \cap C^{\frac{\alpha}{1-\alpha}}([0,1], L^\infty(B)),$
        \item \label{item:v-gluable}$\partial_t^k v(0,\cdot) = \partial_t^k v(1,\cdot) =0$ for all $k \geq 0$,
        \item \label{item:v-periodic} $v$ extends to a continuous periodic function on $\T^2$.
    \end{enumerate}
\end{lemma}

Note that $v \in L^\infty([0,1], C^\alpha(B))$ even though we are compressing a Lipschitz flow on $[0,\infty)$ into a flow on $[0,1/2]$. This heavily utilizes that $\|U(t,\cdot)\|_{L^\infty}$ goes to $0$ quite rapidly in $t$ while the Lipschitz norm stays unit sized. The flow $v$ is now a finite time perfect mixer of the two-cell initial data, as by construction
\[\sol_{0,t_n}^v \Theta_0 = \sol_{0,n}^U \Theta_0 = \Theta_n\]
and so
\[\sol_{0,1/2}^v \Theta_0 = \lim_n \sol_{0,t_n}^v \Theta_0 = \lim_n \Theta_n = \tfrac{1}{2},\]
where for the final equality, we mean the limit in a weak (distributional) sense, and we use Theorem~\ref{thm:ACM_field}, Item~\ref{item:box_averages}.

As it is relevant in the following discussion, we now introduce the notation for averages of functions we will use throughout the paper.

\begin{definition} 
Given a function $f:B\rightarrow \R$, for any $A\subseteq B$
\[(f)_A:=\frac{1}{|A|}\int_A f(x)\,dx.\]
\end{definition}

\subsubsection{Sketch of the proof of Theorem~\ref{thm:two-cell-dissipation}}

Our goal now is to show that the perfect mixing of $v$ is robust to perturbation by diffusion so that we get asymptotic total dissipation of the two-cell initial data given by Theorem~\ref{thm:two-cell-dissipation}. The proof proceeds in three stages, keeping track of the evolution of the data under the advection-diffusion equation over the three time intervals: $[0,t_n], [t_n,\tfrac{1}{2}], $ and $[\tfrac{1}{2},1]$. Here $n$ will be chosen according to $\kappa$ ($n$ is given precisely in Proposition~\ref{prop:close-to-transport}).

On $[0,t_n]$, we want to argue that the solution to the advection-diffusion equation is close to the transport solution, that is
\[\sol_{0,t_n}^{v,\kappa,f} \Theta_0 \approx \Theta_n.\]
The reason this is true is that for this initial time interval, the velocity field $v$ is regular enough and the diffusion $\kappa>0$ is small enough that the diffusion makes only very small errors. This is easy to prove if we choose $n$ to grow quite slowly as $\kappa \searrow 0$. However, we want $5^{-n} \ll \kappa^{1/2}$, as the diffusion is only able to smooth over fluctuations on length scales below $\kappa^{1/2}$ in unit time. As such, we need to be somewhat careful in the analysis. Proposition~\ref{prop:close-to-transport} gives that $\sol_{0,t_n}^{v,\kappa,f} \Theta_0 \approx \Theta_n$ for $5^{-n} \approx \kappa^{\frac{3}{2(\alpha+2)}} \ll \kappa^{1/2}$. The proof of Proposition~\ref{prop:close-to-transport} is essentially elementary and relies on the estimate of stability under diffusion given by Lemma~\ref{lem:small_kappa_difference}. Lemma~\ref{lem:small_kappa_difference}---like the remaining drift-diffusion stability lemmas---is proved in Appendix~\ref{appendix:stability} using the stochastic representation of the drift-diffusion equation.

On $[t_n,\tfrac{1}{2}]$, we can no longer hope that the drift-diffusion equation stays close to the transport equation: eventually such small scales are developed that the diffusive term is dominant. We rely on a much cruder estimate. In order to get diffusion on the time interval $[\tfrac{1}{2},1],$ we require only that 
\[\big(\sol_{0,1/2}^{v,\kappa,f} \Theta_0\big)_{B_m+x_m} \approx\tfrac{1}{2},\qquad \forall x_m \in \Pi_m,\]
for some sufficiently large $m$. That is, we require that $\sol_{0,1/2}^{v,\kappa,f} \Theta_0$ has the right averages on a suitably small length scale. The action of the diffusion on $[\tfrac{1}{2},1]$ will then be to replace $\sol_{0,1/2}^{v,\kappa,f} \Theta_0$ with its average on this small scale.

As such, we do not need to show that the drift-diffusion solution stays close to the transport solution on $[t_n,\tfrac{1}{2}]$, but rather only that the averages on some suitably small scale do not change that much. Since we know that $\sol_{0,t_n}^{v,\kappa,f} \Theta_0 \approx \Theta_n$ and from Theorem~\ref{thm:ACM_field}, Item~\ref{item:box_averages} that 
\[\big(\Theta_n\big)_{B_n+x_n} =\tfrac{1}{2},\qquad \forall x_n \in \Pi_n,\]
we thus get that for all $m \leq n$, 
\[\big(\sol_{0,t_n}^{v,\kappa,f} \Theta_0\big)_{B_m+ x_m} \approx\tfrac{1}{2},\qquad \forall x_m \in \Pi_m.\]
We then choose $m$ so that $5^{-n} \ll 5^{-m} \ll \kappa^{1/2}$ (made precise in Proposition~\ref{prop:advection-noise-small}). Then we show that the average of the solution on boxes of size $5^{-m}$ does not change much on the time interval $[t_n,\tfrac{1}{2}]$. This argument relies essentially on the fact that the typical motion of a particle under the action of the diffusion together with the drift on the time interval $[t_n, \tfrac{1}{2}]$ is $\ll 5^{-m}$ and so the only a small amount of mass could enter or exit a box with sidelength $5^{-m}$. This is shown in Proposition~\ref{prop:advection-noise-small} and relies on the stability estimate given by Lemma~\ref{lem:mean-preservation}. 

The first two steps then give that for some $m$ such that $5^{-m} \ll \kappa^{1/2}$, we have that
\begin{equation}
\label{eq:averages-heuristic}
\big(\sol_{0,1/2}^{v,\kappa,f} \Theta_0\big)_{B_m+ x_m}\approx\tfrac{1}{2},\qquad \forall x_m \in \Pi_m.
\end{equation}
Then on the time interval $[\tfrac{1}{2},1]$, we have that $v =0$, so the solution evolves by pure diffusion. The diffusion essentially averages the solution to length scale $\kappa^{1/2}$, so together with~\eqref{eq:averages-heuristic}, we then get
\[ \sol_{0,1}^{v,\kappa,f} \Theta_0 \approx\tfrac{1}{2},\]
thus concluding the proof of Theorem~\ref{thm:two-cell-dissipation}. 










\subsection{Constructing the universal total dissipator}

\label{subsec:universal-sketch}

\subsubsection{The universal totally dissipating flow}

We now need to introduce notation for the decomposition of the box $B$ by repeatedly splitting it in half, illustrated in Figure~\ref{fig:split}.

\begin{definition}
Let
    \begin{align*}
    \mathcal{R} &:= \frac{1}{\sqrt{2}} \begin{pmatrix} 0 & -1 \\ 1 & 0\end{pmatrix},\\
    A_n&:=\mathcal{R}^nB,\\
    \Lambda_n&:=\mathcal{R}^n\Big\{(\sqrt{2}i,j):0\leq i\leq 2^{\lfloor\frac{n}{2}\rfloor},0\leq j\leq 2^{\lfloor\frac{n+1}{2}\rfloor}\Big\}.
    \end{align*}
\end{definition}

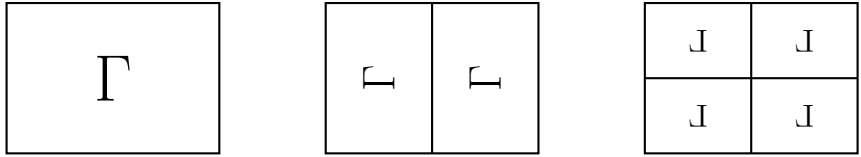
\begin{figure}[htbp]
    \centering
\begin{tikzpicture}

\def\width{2*1.4142}  
\def\height{2}        

\draw[thick] (0, 0) rectangle (\width, \height);

\node at (\width/2, \height/2) {\Huge $\Gamma$};

\draw[thick] (\width + \width/2, 0) rectangle (2*\width + \width/2, \height);
\draw[thick] (\width + \width/2 + \width/2, 0) -- (\width + \width/2 + \width/2, \height);

  \node[rotate=90] at (\width + \width/2 + \width/4, \height/2) {\LARGE $\Gamma$};

  \node[rotate=90] at (\width + \width/2 + 3*\width/4, \height/2) {\LARGE $\Gamma$};

\draw[thick] (3*\width, 0) rectangle (4*\width, \height);

\draw[thick] (3*\width+\width/2, 0) -- (3*\width+\width/2, \height);

\draw[thick] (3*\width,\height/2) -- (4*\width, \height/2);

\node[rotate=180] at (3*\width+\width/4, \height/4) {\large $\Gamma$};

\node[rotate=180] at (3*\width+\width/2+\width/4, \height/4) {\large $\Gamma$};

\node[rotate=180] at (3*\width+\width/4, \height/2+\height/4) {\large $\Gamma$};

\node[rotate=180] at (3*\width+\width/2+\width/4, \height/2+\height/4) {\large $\Gamma$};

\end{tikzpicture}
    \caption{The left box is $A_0 = B$. The middle box is made up of the collection of sub-boxes $(A_1+x_1)_{x_1\in \Lambda_1}$. The right box is made up of the collection of sub-boxes $(A_2+x_2)_{x_2\in \Lambda_2}$. The $\Gamma$ tracks how the orientation of the box transforms under repeated applications of $\mathcal{R}.$}
    \label{fig:split}    
\end{figure}

We can now define our universal totally dissipating flow $V$ as a space and time rescaling of the two-cell dissipating flow $v$.

\begin{definition}\label{def:V}
    Define $\sigma_j :=\frac{1}{2}\big(2^{(\alpha-1)j/2}-2^{(\alpha-1)(j+1)/2}\big)$ and $ s_n := \frac{1}{2} + \sum_{j=n}^\infty\sigma_j.$ Then letting $v$ be given by in~\eqref{eq:v-def} and identifying it with its periodic extension on $\R^2$, we define the divergence-free vector field $V: [0,1] \times \T^2 \to \R^2$ by
    \begin{equation}
        \label{eq:V-def}
        V(t,x) := 
        \begin{cases}
            0 & t \leq 1/2,\\
            \sigma_n^{-1}\mathcal{R}^{n}v(\sigma_n^{-1} (t-s_{n+1}), \mathcal{R}^{-n} x) & t \in [s_{n+1}, s_n].
        \end{cases}
    \end{equation}
\end{definition}

The computational verification of the following is directly analogous to that of Lemma~\ref{lem:v-props}. 

\begin{lemma}
    \label{lem:V-props}
    For $V$ defined by~\eqref{eq:V-def}, we have $V \in L^\infty([0,1], C^\alpha(\T^2)) \cap C^{\frac{\alpha}{1-\alpha}}([0,1], L^\infty(\T^2)).$
\end{lemma}

We note that $V$ is precisely defined as to realize the sketch given by Figure~\ref{fig:total-dissipator} on the time interval $[\tfrac{1}{2}, 1].$

\subsubsection{Sketch of the proof of Theorem~\ref{thm:total_dissipation}}

To prove Theorem~\ref{thm:total_dissipation}, we have to track how the presence of a small diffusion $\kappa \Delta$ creates errors away from the idealized picture given in Figure~\ref{fig:total-dissipator}. The proof of Theorem~\ref{thm:total_dissipation} is in some sense the mirror image of the proof of Theorem~\ref{thm:two-cell-dissipation}. Again we split into three intervals and track the solution separately on each. The time intervals we consider are $[0,\tfrac{1}{2}], [\tfrac{1}{2}, t_n],$ and $[t_n,1]$ with $n $ (different from the $n$ in Section~\ref{sec:two-cell}) chosen according to $\kappa>0.$

On the time interval $[0,\tfrac{1}{2}]$, we have that $V=0$ and so the solution evolves according to pure diffusion. The diffusion then regularizes the solutions to any length scale $2^{-m/2} \ll \kappa^{1/2}$. In particular, up to vanishing error $\sol_{0,1/2}^{V,\kappa,\T^2} \theta_0$ is piecewise constant on the boxes $(A_m + x_m)_{x_m \in \Lambda_m}.$  

On the time interval $[\tfrac{1}{2}, t_n]$, we argue that the property that $\sol_{0,t_n}^{V,\kappa,\T^2} \theta_0$ is piecewise constant on the boxes $(A_m + x_m)_{x_m \in \Lambda_m}$ still holds. This relies on taking $2^{-n/2} \ll 2^{-m/2} \ll \kappa^{1/2}$. Similar to the second step of the argument for Theorem~\ref{thm:two-cell-dissipation}, we rely on the drift and the diffusion typically moving particles $\ll 2^{-m/2}$ and as such the property of being piecewise constant is essentially unaffected. This is proved in Proposition~\ref{prop:stays_near_constant} and relies on the stability estimate given by Lemma~\ref{lem:constant_change}.

Then, on the time interval $[t_n, 1]$, we want to iteratively use the estimate given by Theorem~\ref{thm:two-cell-dissipation}. We have from the previous steps, using that $n \gg m$---so being piecewise constant on $A_m$ boxes implies being piecewise constant on $A_n$ boxes---that $\sol_{0,t_n}^{V,\kappa,\T^2} \theta_0$ is piecewise constant on the boxes $(A_n + x_n)_{x_n \in \Lambda_n}$ up to vanishing error. We can now use the estimate given by Theorem~\ref{thm:two-cell-dissipation} to give that $\sol_{0,t_\ell}^{V,\kappa,\T^2} \theta_0$ is piecewise constant on the boxes $(A_\ell + x_\ell)_{x_\ell\in \Lambda_\ell}$ for all $\ell \leq n$, up to an error that is vanishing provided we choose $n$ correctly. This is proved in Proposition~\ref{prop:2_cell_averaging} and the correct choice of $n$ is given.

Finally, we have that $\sol_{0,t_0}^{V,\kappa,\T^2} \theta_0 = \sol_{0,1}^{V,\kappa,\T^2} \theta_0$ is piecewise constant on $A_0 = B$, i.e.\ $\sol_{0,1}^{V,\kappa,\T^2} \theta_0$ is pointwise equal to its average---again up to vanishing error. This concludes the proof of Theorem~\ref{thm:total_dissipation}.

\subsection{Maximal spreading of stochastic trajectories}

\label{subsec:spread}

This subsection serves to remark on an interesting property of the flow. It is however independent of the rest of the paper and the complete argument is only suggested.

Letting $X^\kappa_t$ solve the SDE~\eqref{eq:stochastic-intro} with $\theta_0(dx) = \delta_y(dx)$, then the proof outlined above implies that 
\[\lim_{\kappa \to 0} \mathcal{L}(X^\kappa_{1/2}) = \delta_y(dx) \quad \text{and} \quad \lim_{\kappa \to 0} \mathcal{L}(X^\kappa_{s_n}) = \frac{1}{|A_n|} \indc_{x \in A_n + x_n}(dx),\]
where $x_n \in \Lambda_n$ is such that $y \in A_n + x_n$ and $\mathcal{L}(X)$ denotes the law of a random variable $X$.

As such, we have that---in the vanishing noise limit---the particle trajectories given by different realizations of $X^\kappa_t$ spread out to the length scale $2^{-n/2}$ on the time interval $[\frac{1}{2},s_n]$. Since $s_n - \frac{1}{2}\approx 2^{-(1-\alpha) n/2}$, we see that in time $t = s_n-\frac{1}{2}$, the particle trajectories spread out to length scale $ \approx t^{\frac{1}{1-\alpha}}$. 

This scaling is consistent with the model case of non-trivial spreading of stochastic trajectories in the vanishing noise limit induced by a $C^\alpha$ vector field, given by taking the drift to be $u(x) =x|x|^{\alpha-1}$. In fact, using that in the vanishing noise asymptotic the solution to the SDE is concentrated on ODE trajectories, one can verify that this spreading is maximal. That is, for a vector field $u \in L^\infty_t C^\alpha_x$, it is impossible to have the particle trajectories spread at a rate faster than $t^{\frac{1}{1-\alpha}}.$

As such, the vector field $V$ causes particles to spread out at the maximal rate (given the spatial regularity) for \textit{every} initial point $y \in \T^2$. In some sense, the vector field behaves as is if there were a $|x-y|^\alpha$ cusp at every $y \in \T^2$.

\subsection{Proof of Corollary~\ref{cor:other-dissipations}}

\begin{proof}[Proof of Corollary~\ref{cor:other-dissipations}]
    We use $L^p_0$ to denote the subspace of $L^p$ given by mean-zero functions, similarly define $H^\sigma_0$. Then we can rephrase Theorem~\ref{thm:total_dissipation} as an operator norm bound:
    \[\Big\|\sol_{0,1}^{\kappa,V,\T^2}\Big\|_{L^1_0 \to L^1_0} \leq C \kappa^{\frac{(1-\alpha)^2}{72}}.\]
    Taking the adjoint and using that the adjoint solution operator is the solution operator for the backward flow, we get
    \[\Big\|\sol_{0,1}^{\kappa,\tilde V,\T^2}\Big\|_{L^\infty_0 \to L^\infty_0} \leq C \kappa^{\frac{(1-\alpha)^2}{72}},\]
    where $\tilde V(t,x) := V(1-t,x).$

    Then using that
    \[\sol_{0,1}^{\kappa, W, \T^2} =  e^{\kappa \Delta/4}\sol_{0,1}^{\kappa/4, \tilde V, \T^2}\sol_{0,1}^{\kappa/4, V, \T^2} e^{\kappa \Delta/4}\]
    and that the solution operator for a drift-diffusion equation is a contraction in $L^1$ and $L^\infty$, we get 
    \[\Big\|\sol_{0,1}^{\kappa, W, \T^2}\Big\|_{L^1_0 \to L^1_0} \leq  \Big\|\sol_{0,1}^{\kappa/4, \tilde V, \T^2}\Big\|_{L^1_0 \to L^1_0} \Big\|\sol_{0,1}^{\kappa/4, V, \T^2}\Big\|_{L^1_0 \to L^1_0}  \leq C \kappa^{\frac{(1-\alpha)^2}{72}},\]
    and similarly
    \[\Big\|\sol_{0,1}^{\kappa, W, \T^2}\Big\|_{L^\infty_0 \to L^\infty_0} \leq  C \kappa^{\frac{(1-\alpha)^2}{72}}.\]
    We then conclude~\eqref{eq:cor-lp} by Riesz-Thorin interpolation.

    For~\eqref{eq:cor-Hs}, we compute using standard smoothing estimates from the heat flow that
    \[
    \|\sol_{0,1}^{\kappa, W, \T^2}\|_{H^{-\sigma}_0 \to H^\sigma_0} \leq \|e^{\kappa  \Delta/4}\|_{L^2_0 \to H^\sigma_0} \|\sol_{1/4,3/4}^{\kappa, W,\T^2}\|_{L^2_0 \to L^2_0} \|e^{\kappa \Delta/4}\|_{H^{-\sigma}_0 \to L^2_0} 
    \leq C \kappa^{-2\sigma}\kappa^{\frac{(1-\alpha)^2}{72}},\]
    and then conclude for the range of $\sigma$ under consideration.
\end{proof}

\section{Total dissipation of two-cell data}

\label{sec:two-cell}

In the arguments below, some attempt at maintaining optimality in $\alpha$ is made. For example, we balance the size of the error in each step of the proof when choosing $n,m$ such that $5^{-n} \ll 5^{-m} \ll \kappa^{1/2}$. However, we do not keep the constant prefactor in the rate optimal for the sake of simplicity.

Note that all lemmas stated below, as well as in Section~\ref{sec:universal}, have their proofs deferred to Appendix~\ref{appendix:stability}. We do this in order to separate the relevant facts about drift-diffusion equations from the specific argument used to prove total dissipation.

    The following proposition controls the evolution of the two-cell data under the two-cell dissipating flow on the time interval $[0,t_n]$, as discussed in Subsection~\ref{subsec:two-cell-sketch}. Note particularly that we get an algebraic rate of error $\kappa^{(1-\alpha)/12}$ that degenerates to $0$ as $\alpha \to 1$, in line with anomalous dissipation being impossible at Lipschitz regularity. Note also that the length scale up to which we are treating the diffusion as a perturbative error is given by $5^{-n} \approx \kappa^{\frac{3}{2(\alpha+2)}} \ll \kappa^{1/2}$ with again the separation degenerating as $\alpha \to 1$.

\begin{proposition}
\label{prop:close-to-transport}
    There exists $C(\alpha)>0$ so that for all boundary data $f\in L^\infty([0,1]\times\partial B)$ and $\kappa>0$,  letting $n:=\big\lceil\tfrac{3}{2(\alpha+2)}\log_5(\kappa^{-1})\big\rceil$, it holds that 
\[\big\|\sol_{0,t_n}^{v,\kappa,f} \Theta_0-\Theta_n\big\|_{L^1(B)}\leq C\big(1 +\|f\|_{L^\infty([0,1]\times\partial B)}\big)\kappa^{(1-\alpha)/12}.\]
\end{proposition}

The following lemma bounds the error between the drift-diffusion equation and the pure transport equation utilizing ``regularity'' of the initial data $\indc_E$ through the measure of the boundary of the set $E$ (which e.g.\ controls BV regularity). This lemma---suitably rescaled in time---is used iteratively as the main tool of Proposition~\ref{prop:close-to-transport}.

\begin{lemma}\label{lem:small_kappa_difference}
    For any divergence-free vector field $u \in L^\infty\big([0,1],W^{1,\infty}(B)\big)$ tangent to $\partial B$, there exists a constant $C\big(\|u\|_{L^\infty([0,1],W^{1,\infty}(B))}\big)>0$ so that for all boundary data $f\in L^\infty([0,1]\times \partial B)$, $\kappa>0$, and subsets $E\subseteq B$ so that $\partial E$ can be parametrized by a finite length Lipschitz curve, we have the estimate
    \[\big\|\sol_{0,1}^{u,\kappa,f} \indc_E - \sol_{0,1}^{u}\indc_E\big\|_{L^1(B)} \leq C\big(1+\|f\|_{L^\infty([0,1]\times\partial B)}\big)(1+|\partial E|)\kappa^{1/2} \log (\kappa^{-1}),\]
    where $|\partial E|$ denotes the length of the boundary curve for $\partial E$.
\end{lemma}

\begin{proof}[Proof of Proposition~\ref{prop:close-to-transport}]
Rewriting the difference between $\sol^{v,\kappa,f}_{0,t_n}\Theta_0$ and $\Theta_n$ as a telescoping sum, we have that
\begin{equation*}
\big\|\sol^{v,\kappa,f}_{0,t_n}\Theta_0-\Theta_n\big\|_{L^1}=\bigg\|\sum_{j=0}^{n-1}\sol^{v,\kappa,f}_{t_j,t_n}\Theta_{j}-\sol^{v,\kappa,f}_{t_{j+1},t_n}\Theta_{j+1}\bigg\|_{L^1}\leq \sum_{j=0}^{n-1}\bigg\|\sol^{v,\kappa,f}_{t_j,t_n}\Theta_{j}-\sol^{v,\kappa,f}_{t_{j+1},t_n}\Theta_{j+1}\bigg\|_{L^1},
\end{equation*}
where the last line follows by the triangle inequality. Since the solution operator $\sol^{v,\kappa,0}_{s,t}:L^1(B)\rightarrow L^1(B)$ is contractive for all $
s\leq t$,
\begin{equation}\label{eq:telescope}
\sum_{j=0}^{n-1}\bigg\|\sol^{v,\kappa,f}_{t_j,t_n}\Theta_{j}-\sol^{v,\kappa,f}_{t_{j+1},t_n}\Theta_{j+1}\bigg\|_{L^1}=\sum_{j=0}^{n-1}\bigg\|\sol_{t_{j+1},t_n}^{v,\kappa,0}\Big(\sol^{v,\kappa,f}_{t_j,t_{j+1}}\Theta_{j}-\Theta_{j+1}\Big)\bigg\|_{L^1}\leq \sum_{j=0}^{n-1}\big\|\sol^{v,\kappa,f}_{t_j,t_{j+1}}\Theta_{j}-\Theta_{j+1}\big\|_{L^1}.
\end{equation}
We will bound each term in the sum on the right-hand side above using Lemma~\ref{lem:small_kappa_difference}.

For any fixed $j$, letting $g(t,x):=f(\tau_j(t-j)+t_j,x)$, by rescaling in time it holds that
\[\sol^{v,\kappa,f}_{t_j,t_{j+1}}\Theta_j=\sol^{U,\tau_j\kappa,g}_{j,j+1}\Theta_j.\]
Since $\sol^U_{j,j+1}\Theta_j=\sol^U_{j,j+1}\indc_{E_j}=\indc_{E_{j+1}}=\Theta_{j+1}$ by Item~\ref{item:set_mapping}, Lemma~\ref{lem:small_kappa_difference} implies that
\[\big\|\sol^{v,\kappa,f}_{t_j,t_{j+1}}\Theta_{j}-\Theta_{j+1}\big\|_{L^1}=\Big\|\sol^{U,\tau_j\kappa,g}_{j,j+1}\indc_{E_j}-\sol^U_{j,j+1}\indc_{E_j}\Big\|_{L^1}\leq C(1+\|g\|_{L^\infty})(1+|\partial E_j|)(\tau_j\kappa)^{1/2} \log\big( (\tau_j\kappa)^{-1}\big),\]
for some constant $C>0$ that only depends on $\|U\|_{L^\infty_t W^{1,\infty}_x}$. Since $\|g\|_{L^\infty}\leq \|f\|_{L^\infty}$ and $|\partial E_j|\leq C5^j$ by Item~\ref{item:boundary_curves}, the above inequality implies that
\begin{equation}\label{eq:j_to_j+1_bound}
\big\|\sol^{v,\kappa,f}_{t_j,t_{j+1}}\Theta_{j}-\Theta_{j+1}\big\|_{L^1}\leq C(1+\|f\|_{L^\infty})5^j(\tau_j\kappa)^{1/2} \log\big((\tau_j\kappa)^{-1}\big)
\end{equation}
for all $j$.

Inserting~\eqref{eq:j_to_j+1_bound} into~\eqref{eq:telescope}, we have in total found that
\[\|\sol^{v,\kappa,f}_{0,t_n}\Theta_0-\Theta_n\|_{L^1}\leq C(1+\|f\|_{L^\infty})\sum_{j=0}^{n-1}5^j(\tau_j\kappa)^{1/2}\log \big((\tau_j\kappa)^{-1}\big).\]
Uniformly bounding $\log \big((\tau_j\kappa)^{-1}\big)$ by $\log \big((\tau_n\kappa)^{-1}\big)$ and using that $\sum_{j=0}^{n-1}5^j\tau_j^{1/2}\leq C5^n\tau_n^{1/2}$ for some $C>0$, we have shown that
\[\big\|\sol^{v,\kappa,f}_{0,t_n}\Theta_0-\Theta_n\big\|_{L^1}\leq C(1+\|f\|_{L^\infty})5^n(\tau_n\kappa)^{1/2}\log ((\tau_{n}\kappa)^{-1}).\]
By our choice of $n$, $5^n(\tau_n\kappa)^{1/2}\leq C\kappa^{\frac{(1-\alpha)}{4(\alpha+2)}}$ and $\log( (\tau_n\kappa)^{-1})\leq C\log(\kappa^{-1})$, thus the proposition holds since $\kappa^{\frac{(1-\alpha)}{4(\alpha+2)}}\log(\kappa^{-1})\leq C\kappa^{(1-\alpha)/12}$.
\end{proof}

The following proposition controls the evolution of the two-cell data under the two-cell dissipating flow on the time interval $[t_n,\tfrac{1}{2}]$, as discussed in Subsection~\ref{subsec:two-cell-sketch}. Recall that our goal is to show that the solution has the correct means on small boxes $B_m + x_m$ at time $1/2$, that is $\big(\sol_{0,1/2}^{v,\kappa,f} \Theta_0)_{B_m + x} \approx \tfrac{1}{2}$ for all $x_m \in \Pi_m$, where $m$ is such that $5^{-n} \ll 5^{-m} \ll \kappa^{1/2}.$ The way this is made precise below is to say that $\sol_{0,1/2}^{v,\kappa,f} \Theta_0$ is $L^1$ close to a function $\varphi$, where $\varphi$ has the exact right means: $(\varphi)_{B_m + x} =\tfrac{1}{2}.$ Note by Proposition~\ref{prop:close-to-transport}, we already have that $\sol_{0,1/2}^{v,\kappa,f} \Theta_0 \approx \sol_{t_n,1/2}^{v,\kappa,f} \Theta_n$. As such, we define $\varphi$ exactly as $\sol_{t_n,1/2}^{v,\kappa,f} \Theta_n$ but with the means corrected: we add a function that is piecewise constant on the $5^{-m}$ grid that corrects the means to be exactly $\tfrac{1}{2}$ on each box $B_m + x_m, x_m \in \Pi_m$, as is defined in~\eqref{eq:varphi-def}.

\begin{proposition}
    \label{prop:advection-noise-small}
    There exists $C(\alpha)>0$ so that for all boundary data $f\in L^\infty([0,1]\times\partial B)$ and $\kappa>0$, letting
    \[n:=\big\lceil\tfrac{3}{2(\alpha+2)}\log_5(\kappa^{-1})\big\rceil\text{ and } m :=\big\lceil\tfrac{\alpha+5}{4(\alpha+2)}\log_5(\kappa^{-1})\big\rceil,\]
    there exists $\varphi \in L^\infty(B)$ so that $\|\varphi \|_{L^\infty(B)}\leq C\big(1 +\|f\|_{L^\infty([0,1]\times\partial B)}\big)$, $(\varphi)_{B_m+x_m}=\frac{1}{2}$ for all $x_m\in\Pi_m$, and
\[\Big\|\sol^{v,\kappa,f}_{t_n,1/2}\Theta_n-\varphi\Big\|_{L^1(B)}\leq C\big(1 + \|f\|_{L^\infty([0,1]\times\partial B)}\big)\kappa^{(1-\alpha)/12}.\]
\end{proposition}

    The following lemma controls how much the drift and diffusion can change the mean of the initial data. Note particularly the presence of a $L^1_tL^\infty_x$ norm on $u$ on the right-hand side. This bound---suitably rescaled in space and time---is the central tool of the proof of Proposition~\ref{prop:advection-noise-small}.  

\begin{lemma}
    \label{lem:mean-preservation}
    There exists $C>0$ so that for any divergence-free vector field $u \in L^1\big([0,1],L^\infty(B)\big)$, $\kappa>0$, boundary data $f\in L^\infty([0,1]\times\partial B)$, and initial datum $\theta_0$, it holds that
    \[\Big|\big(\sol_{0,1}^{u,\kappa,f} \theta_0\big)_B- \big(\theta_0\big)_B\Big| \leq C\big(\|\theta_0\|_{L^\infty(B)}+\|f\|_{L^\infty([0,1]\times\partial B)}\big)\big(\|u\|_{L^1([0,1],L^\infty(B))}+\kappa^{1/2} \log (\kappa^{-1})\big).\]
\end{lemma}

\begin{proof}[Proof of Proposition~\ref{prop:advection-noise-small}] 
We let
\begin{equation}
\label{eq:varphi-def}
\varphi(x):=\sol^{v,\kappa,f}_{t_n,1/2}\Theta_n(x)+\sum_{x_m\in\Pi_m}\bigg(\tfrac{1}{2}-\Big(\sol^{v,\kappa,f}_{t_n,1/2}\Theta_n\Big)_{B_m+x_m}\bigg)\indc_{B_m+x_m}(x)
\end{equation}
    so that $(\varphi)_{B_m+x_m}=\tfrac{1}{2}=\big(\Theta_n\big)_{B_m+x_m}$ for all $x_m\in \Pi_m$ by Item~\ref{item:box_averages}. We thus have that
\begin{align}
\big\|\sol^{v,\kappa,f}_{t_n,1/2}\Theta_n-\varphi\big\|_{L^1(B)}&=\sum_{x_m\in\Pi_m}\int_{B_m+x_m}\bigg|\big(\Theta_n\big)_{B_m+x_m} -\big(\sol^{v,\kappa,f}_{t_n,1/2}\Theta_n\Big)_{B_m+x_m}\bigg|\,dx\notag
\\&\leq \max_{x_m\in \Pi_m}\bigg|\big(\Theta_n\big)_{B_m+x_m} -\big(\sol^{v,\kappa,f}_{t_n,1/2}\Theta_n\Big)_{B_m+x_m}\bigg|.\label{eq:box_averages}
\end{align}
We just need to bound the term in the max in~\eqref{eq:box_averages} uniformly over all $x_m\in\Pi_m$.

Letting $\tau := \frac{1}{2} - t_n,$ $g(t,x):=\sol_{t_n,t_n+\tau t}^{v,\kappa,f}\Theta_n(5^{-m}x+x_m)$, $w(t,x):=\tau 5^mv(\tau t+t_n,5^{-m}x+x_m)$,
and $\Psi(x):=\Theta_n(5^{-m}x+x_m),$ by rescaling in time and space it holds that
\[\sol_{t_n,1/2}^{v,\kappa,f}\Theta_n(5^{-m}x+x_m)=\sol_{0,1}^{w,\tau5^{2m}\kappa,g}\Psi(x)\]
for all $x\in B$. Together with Lemma~\ref{lem:mean-preservation} this implies that for each $x_m \in \Pi_m$
\begin{align*}\bigg|\big(\Theta_n\big)_{B_m+x_m} -\big(\sol^{v,\kappa,f}_{t_n,1/2}\Theta_n\Big)_{B_m+x_m}\bigg|&=\bigg| \big( \sol^{w,\tau5^{2m}\kappa,g}_{0,1}\Psi\big)_B-\big(\Psi\big)_B \bigg|
\\&\leq C(\|\Psi\|_{L^\infty}+\|g\|_{L^\infty})\big(\|w\|_{L^1_tL^\infty_x}+(\tau5^{2m}\kappa)^{1/2} \log\big( (\tau5^{2m}\kappa)^{-1}\big)\big).
\end{align*}
Clearly $\|\Psi\|_{L^\infty}\leq 1$ and $\|g\|_{L^\infty}\leq 1+\|f\|_{L^\infty}.$ Using the definitions of $w$ and $v$ with Item~\ref{item:L^infty_bound} in Theorem~\ref{thm:ACM_field} we see that
\[\|w\|_{L^1_tL^\infty_x}\leq5^{m}\sum_{j=n}^\infty\tau_j\|v\|_{L^\infty([t_n,t_{n+1}]\times B)}\leq5^{m}\sum_{j=n}^\infty\|U\|_{L^\infty([n,n+1]\times B)}\leq C5^{m-n}.\]
Combining these bounds, in total we have found that
\begin{equation}\label{eq:complicated_bound}
\big\|\sol^{v,\kappa,f}_{t_n,1/2}\Theta_n-\varphi\big\|_{L^1(B)}\leq C(1+\|f\|_{L^\infty})\big(5^{m-n}+(\tau5^{2m}\kappa)^{1/2} \log \big((\tau5^{2m}\kappa)^{-1}\big)\big).    
\end{equation}

To conclude, it thus suffices to bound the last factor in the right-hand side of~\eqref{eq:complicated_bound} by some constant times $\kappa^{(1-\alpha)/12}$. Since $\tau=\sum_{j=n}^\infty\tau_j\leq C\tau_n$, and $m<n$, we have the immediate bound
\[(\tau5^{2m}\kappa)^{1/2} \log \big((\tau5^{2m}\kappa)^{-1}\big)\leq C(\tau_n5^{2n}\kappa)^{1/2} \log ((\tau_n\kappa)^{-1}).\]
In Proposition~\ref{prop:close-to-transport} we bounded the term on the right-hand side above by $C\kappa^{(1-\alpha)/12}$. On the other hand, by our choice of $n$ and $m$ 
\[5^{m-n}\leq C \kappa^{\frac{(1-\alpha)}{4(\alpha+2)}}\leq C\kappa^{(1-\alpha)/12}.\]
We have thus appropriately bounded each term.
\end{proof}

Finally, we want to control the evolution of the solution on the interval $[\frac{1}{2},1]$. On this interval, $v =0$ and as such we are undergoing pure diffusion under the heat equation. The following lemma states that the diffusion under the heat equation on the box is equivalent to the diffusion on the torus up to a small error. This allows us to use the explicit form of the heat kernel on $\T^2$ in the proof of Theorem~\ref{thm:two-cell-dissipation}.

\begin{lemma}
    \label{lem:close-to-toroidal-heat}
   There exists $C>0$ so that for all boundary data $f\in L^\infty([0,1]\times\partial B)$, initial data $\theta_0 \in L^\infty(B),$ and $\kappa>0$, it holds that
    \[\big\|\sol_{0,1}^{0,\kappa,f} \theta_0 - \sol_{0,1}^{0,\kappa,\T^2} \theta_0\big\|_{L^1(B)} \leq C\big(\|\theta_0\|_{L^\infty(B)} + \|f\|_{L^\infty([0,1]\times\partial B)}\big)\kappa^{1/2} \log (\kappa^{-1}).\]
\end{lemma}

We now conclude Theorem~\ref{thm:two-cell-dissipation}. From the previous propositions, we know that (up to a vanishing error) $\sol_{0,1/2}^{v,\kappa,f} \Theta_0$ has the correct averages on each box $B_m + x_m, x_m \in \Pi_m$. Thus to conclude, we use straightforward bounds on the heat kernel on $\T^2$ (together with Lemma~\ref{lem:close-to-toroidal-heat}) to show that the diffusion essentially averages the solution on scale $5^{-m}$, thus giving that the solution at the final time is (asymptotically) equal to the constant $\frac{1}{2}.$

\begin{proof}[Proof of Theorem~\ref{thm:two-cell-dissipation}] Let $n$, $m$ and $\varphi$ be as in Propositions~\ref{prop:advection-noise-small}. Then the triangle inequality and the fact that the operator $\sol_{s,t}^{v,\kappa,0}$ is an $L^1$ contraction imply that
\[\big\|\sol_{0,1}^{v,\kappa,f} \Theta_0 - \tfrac{1}{2}\big\|_{L^1(B)}\leq \big\|\sol_{0,t_n}^{v,\kappa,f} \Theta_0 - \Theta_n\big\|_{L^1(B)}+\Big\|\sol_{t_n,1/2}^{v,\kappa,f} \Theta_n - \varphi\Big\|_{L^1(B)}+\Big\|\sol_{1/2,1}^{v,\kappa,f} \varphi - \tfrac{1}{2}\Big\|_{L^1(B)}.\]
Propositions~\ref{prop:close-to-transport} and~\ref{prop:advection-noise-small} imply that the first two terms on the right-hand side above are both bounded by some constant times $(1+\|f\|_{L^\infty})\kappa^{(1-\alpha)/12}.$ Thus, to conclude the theorem, it only remains to bound the last term.

Since $v=0$ when $t\in\big[\tfrac{1}{2},1\big]$, Lemma~\ref{lem:close-to-toroidal-heat} implies that
\[\big\|\sol_{1/2,1}^{v,\kappa,f} \varphi - \tfrac{1}{2}\big\|_{L^1(B)}\leq \Big\|\sol_{1/2,1}^{0,\kappa,\T^2} \varphi - \tfrac{1}{2}\Big\|_{L^1(B)}+C(1+\|f\|_{L^\infty})\kappa^{1/2}\log(\kappa^{-1}).\]
Since $\kappa^{1/2}\log(\kappa^{-1})\leq C\kappa^{(1-\alpha)/12},$ to conclude it suffices to show 
\[\Big\|\sol_{1/2,1}^{0,\kappa,\T^2} \varphi - \tfrac{1}{2}\Big\|_{L^1(B)} \leq  C\big(1 +\|f\|_{L^\infty([0,1]\times\partial B)}\big)\kappa^{(1-\alpha)/12}.\] We proceed by letting $\tilde\varphi:=\varphi-1/2$ so that $(\tilde\varphi)_{B_m+x_m}=0$ for all $x_m\in\Pi_m$ and
\[\Big\|\sol_{1/2,1}^{0,\kappa,\T^2}\varphi-\tfrac{1}{2}\Big\|_{L^1}=\|\Phi_{\kappa/2}*\tilde{\varphi}\|_{L^1},\]
where $\Phi_{t}$ denotes the fundamental solution to the heat equation on the torus. 

Expanding out the definition of the convolution we have that
\[\|\Phi_{\kappa/2}*\tilde{\varphi}\|_{L^1}=\sum_{x_m\in \Pi_m}\int_{x_m+B_m}\Bigg|\sum_{y_m\in\Pi_m}\int_{y_m+B_m}\Phi_{\kappa/2}(x-y)\tilde\varphi(y)\,dy\Bigg|\,dx.\]
Since $\tilde\varphi$ is has zero mean on $y_m+B_m$ for all $y_m \in \Pi_m$, this is equal to
\[\sum_{x_m\in \Pi_m}\int_{x_m+B_m}\Bigg|\sum_{y_m\in\Pi_m}\int_{y_m+B_m}\big(\Phi_{\kappa/2}(x-y)-(\Phi_{\kappa/2})_{x_m-y_m+B_{m-1}}\big)\tilde\varphi(y)\,dy\Bigg|\,dx,\]
which is in turn bounded by
\begin{equation}\label{eq:Holder_bound}
\sum_{x_m,y_m\in\Pi_m}\int_{y_m+B_m}|\tilde\varphi(y)|\int_{x_m+B_m}\big|\Phi_{\kappa/2}(x-y)-(\Phi_{\kappa/2})_{x_m-y_m+B_{m-1}}\big|\,dx\,dy.\end{equation}
Using that $y\in y_m+B_m$ and the Poincar\'e inequality
\begin{align*}
\int_{x_m+B_m}\big|\Phi_{\kappa/2}(x-y)-(\Phi_{\kappa/2})_{x_m-y_m+B_{m-1}}\big|\,dx&\leq \int_{x_m-y_m+B_{m-1}}\big|\Phi_{\kappa/2}(x)-(\Phi_{\kappa/2})_{x_m-y_m+B_{m-1}}\big|\,dx
\\&\leq C5^{-m}\int_{x_m-y_m+B_{m-1}}|\nabla\Phi_{\kappa/2}(x)|\,dx,
\end{align*}
where we note that the integrals on the right-hand side are over the larger box $B_{m-1}$. Inserting this into~\eqref{eq:Holder_bound}, we thus find that
\begin{align*}
\Big\|\sol_{\frac{1}{2},1}^{0,\kappa,\T^2}\varphi-\tfrac{1}{2}\Big\|_{L^1}&\leq C5^{-m}\sum_{x_m,y_m\in\Pi_m}\int_{y_m+B_m}|\tilde\varphi(y)|\int_{x_m-y_m+B_{m-1}}\big|\nabla\Phi_{\kappa/2}(x)\big|\,dx\,dy
\\&= C5^{-m}\|\nabla\Phi_{\kappa/2}\|_{L^1}\|\tilde\varphi\|_{L^1}
\\&\leq C5^{-m}\kappa^{-1/2}\|\tilde\varphi\|_{L^1}.
\end{align*}
By construction, $\|\tilde\varphi\|_{L^1}\leq C(1+\|f\|_{L^\infty})$ while
\[5^{-m}\kappa^{-1/2}\leq C\kappa^{\frac{(1-\alpha)}{4(\alpha+2)}}\leq C\kappa^{(1-\alpha)/12}.\]
This concludes the theorem.
\end{proof}

We lastly note the following corollary---whose proof is direct from Theorem~\ref{thm:two-cell-dissipation}---stated in the form we will need in Section~\ref{sec:universal}. 

\begin{corollary}\label{cor:two-cell-dissipation}
There exists $C(\alpha)>0$ so that for all boundary data $f\in L^\infty([0,1]\times \partial B)$ and $a_1,a_0\in\R$, letting $\theta_0:=(a_1-a_0)\Theta_0+a_0$ it holds that
\[\big\|\sol_{0,1}^{v,\kappa,f}\theta_0-\tfrac{a_0+a_1}{2}\big\|_{L^1(B)}\leq C\big(|a_1-a_0| +\|f-a_0\|_{L^\infty([0,1]\times \partial B)}\big)\kappa^{(1-\alpha)/12}.\]
\end{corollary}


\section{Universal total dissipation}

\label{sec:universal}

We now move on to the proof of Theorem~\ref{thm:total_dissipation}, following closely the sketch given in Subsection~\ref{subsec:universal-sketch}. There is however an additional technical issue that we did not address in the proof sketch. Note that the central step of the argument relies on iteratively using Theorem~\ref{thm:two-cell-dissipation} on sub-boxes, for which one needs to control the $L^\infty$ norm of the solution on the boundaries of the sub-boxes. However, Theorem~\ref{thm:total_dissipation} is stated for $TV$ initial data, so we do not immediately get $L^\infty$ control of the solution. Using that $V = 0$ on the interval $[0,\frac{1}{2}]$ together with estimates on the heat kernel, we get $\|\sol_{0,1/2}^{V,\kappa,\T^2} \theta_0\|_{L^\infty} \leq C \kappa^{-1/2} \|\theta_0\|_{L^1}.$ Unfortunately, the rate from Theorem~\ref{thm:two-cell-dissipation}, $\kappa^{(1-\alpha)/12}$, cannot compete with $\kappa^{-1/2}$, so we have to be more careful than brutally using the $L^1 \to L^\infty$ smoothing of the heat equation.

This is accomplished with the following lemma. The pure diffusion acting on the interval $[0,\frac{1}{2}]$ causes $\sol_{0,1/2}^{V,\kappa,\T^2} \theta_0$ to be $L^1$ close to a function that is piecewise constant on a grid with sidelength $\ll \kappa^{1/2}$. Throughout the interval $[\frac{1}{2},1]$---as described in the sketch given in Subsection~\ref{subsec:universal-sketch}---$\sol_{0,t}^{V,\kappa,\T^2}$ will continue to be $L^1$ close to a piecewise constant function. The lemma below gives that local $L^\infty$ norms of solutions to drift-diffusion equations are controlled locally by the $L^\infty$ norm of the initial data on a slightly larger box together with an exponential discount on the $L^\infty$ norm of ``far'' initial data. This allows us to exploit the local smoothness given by the solution being (essentially) piecewise constant to control the local $L^\infty$ norm on the right-hand side of the estimate and use the brutal $\kappa^{-1/2}$ bound to control the second term, as the exponential discount dominates.

\begin{lemma}\label{lem:box_bounds}
   There exists $C>0$ so that for any divergence-free vector field $u \in L^1\big([0,1], L^\infty(\T^2)\big)$, initial datum $\theta_0 \in L^\infty(\T^2)$, $\kappa>0, \ell \in (0,1/2]$, and $L\geq 2 \|u\|_{L^1_t L^\infty_x},$ it holds that
    \[\big\|\sol^{u,\kappa,\T^2}_{0,1} \theta_0\big\|_{L^\infty(\sqrt{2}[-\ell,\ell]\times[-\ell,\ell])} \leq \|\theta_0\|_{L^\infty(\sqrt{2}[-\ell,\ell]\times[-\ell,\ell]+[-L,L]^2)} +C \exp(-C^{-1}\kappa^{-1} L^2)\|\theta_0\|_{L^\infty(\T^2)}.\]
\end{lemma}

In the proposition below, we control the solution on the time interval $[t_n,1]$. It is on this interval that we iteratively use Theorem~\ref{thm:two-cell-dissipation} (or more precisely Corollary~\ref{cor:two-cell-dissipation}) to give that---provided the solution is piecewise constant on the $A_n + x_n, x_n \in \Lambda_n$ sub-boxes at time $t_n$---the solution is iteratively averaged to be (essentially) constant on $\T^2$ at time $1$. Note particularly that the $L^\infty$ norm on the right-hand side comes paired with an exponential discount, as discussed above.

In Proposition~\ref{prop:2_cell_averaging}, which in some sense mirrors Proposition~\ref{prop:close-to-transport} of Section~\ref{sec:two-cell}, we take $n$ (essentially) as in Proposition~\ref{prop:close-to-transport}. This is slightly suboptimal but appreciably simplifies statements and computations at the cost of a constant factor in the algebraic rate.

\begin{proposition}
\label{prop:2_cell_averaging}
    There exists $C(\alpha)>0$ so that for all $\kappa>0$, letting $n:=\big\lceil\frac{3}{2(\alpha+2)}\log_{\sqrt{2}}(\kappa^{-1})\big\rceil$, if $\varphi\in L^1(\T^2)$ is constant on $A_n+x_n$ for all $x_n\in\Lambda_n$, then
\[\Big\|\sol_{s_n,1}^{V,\kappa,\T^2}\varphi-(\varphi)_{\T^2}\Big\|_{L^1(\T^2)}\leq C\Big(\|\varphi\|_{L^1(\T^2)}+\|\varphi\|_{L^\infty(\T^2)}\exp\big({-C^{-1}}\kappa^{-(1-\alpha)/6}\big)\Big)\kappa^{\frac{(1-\alpha)^2}{72}}.\]
\end{proposition}

\begin{proof}
For $0\leq j\leq n$,  we let $a_{x_j}:=(\varphi)_{A_j+x_j}$ for all $x_j\in \Lambda_j$ and
\[\varphi_j(x):=\sum_{x_j\in \Lambda_j}a_{x_j}\indc_{A_j+x_j}(x).\]
Thus $\varphi_n=\varphi$ by assumption and $\varphi_0=(\varphi)_{\T^2}$. Rewriting the difference between $\sol_{s_n,1}^{V,\kappa,\T^2}\varphi$ and $(\varphi)_{\T^2}$ as a telescoping sum and using that $\sol_{s,t}^{V,\kappa,\T^2}$ is an $L^1$ contraction we find that
\begin{equation}
\Big\|\sol_{s_n,1}^{V,\kappa,\T^2}\varphi-(\varphi)_{\T^2}\Big\|_{L^1}=\Bigg\|\sum_{j=0}^{n-1}\sol_{s_{j+1},1}^{V,\kappa,\T^2}\varphi_{j+1}-\sol_{s_j,1}^{V,\kappa,\T^2}\varphi_{j}\Bigg\|_{L^1} \leq\sum_{j=0}^{n-1}\Big\|\sol_{s_{j+1},s_j}^{V,\kappa,\T^2}\varphi_{j+1}-\varphi_{j}\Big\|_{L^1}.\label{eq:time_splits}
\end{equation}
Fixing $j$ and splitting the integral over boxes at scale $2^{-j/2}$
\begin{equation}\label{eq:small_box_split}
\Big\|\sol_{s_{j+1},s_j}^{V,\kappa,\T^2}\varphi_{j+1}-\varphi_{j}\Big\|_{L^1(\T^2)}=\sum_{x_j\in\Lambda_j}\big\|\sol_{s_{j+1},s_j}^{V,\kappa,\T^2}\varphi_{j+1}-\varphi_{j}\big\|_{L^1(A_j+x_j)}.
\end{equation}
We will bound each term in the above sum using Corollary~\ref{cor:two-cell-dissipation}.

Letting $\tilde\varphi(x):=\varphi_{j+1}(\mathcal{R}^jx+x_j)$ and 
$f_{x_j}(t,x):=\sol_{s_{j+1},(s_{j+1}+\sigma_jt)}^{V,\kappa,\T^2}\varphi_{j+1}(\mathcal{R}^jx+x_j)$ for any $x_j\in \Lambda_j$, by rescaling in time and space it holds that
\[\sol_{s_{j+1},s_j}^{V,\kappa,\T^2}\varphi_{j+1}(R^jx+x_j)=\sol_{0,1}^{v,\sigma_j2^j\kappa,f_{x_j}}\tilde\varphi(x)\]
for all $x\in B$. By the definition of $\varphi_{j+1}$, there exists $a_{x_j,1},a_{x_j,0}\in\R$ so that
\[\tilde\varphi=(a_{x_j,1}-a_{x_j,0})\Theta_0+a_{x_j,0}\]
and $a_{x_j}=\tfrac{1}{2}\big(a_{x_j,1}+a_{x_j,2}\big).$ Since $\varphi_j$ is constantly equal to $a_{x_j}$ on $A_j+x_j$, Corollary~\ref{cor:two-cell-dissipation} implies that
\begin{align*}
\|\sol_{s_{j+1},s_j}^{V,\kappa,\T^2}\varphi_{j+1}-\varphi_{j}\|_{L^1(A_j+x_j)}&=2^{-j}\|\sol_{0,1}^{v,\sigma_j2^j\kappa,f_{x_j}}\tilde\varphi(x)-a_{x_j}\|_{L^1(B)}
\\&\leq C2^{-j}(|a_{x_j,1}-a_{x_j,0}| +\|f_{x_j}-a_{x_j,0}\|_{L^\infty})(\sigma_j2^j\kappa)^{(1-\alpha)/12}\notag
\\&\leq C2^{-j}(|a_{x_j,1}|+|a_{x_j,0}|+\|f_{x_j}\|_{L^\infty})(\sigma_j2^j\kappa)^{(1-\alpha)/12}.
\end{align*}
Combining this with~\eqref{eq:small_box_split}, we've found that
\begin{equation}\label{eq:j_to_j+1_bound_2}
\Big\|\sol_{s_{j+1},s_j}^{V,\kappa,\T^2}\varphi_{j+1}-\varphi_{j}\Big\|_{L^1(\T^2)}\leq C\bigg(\sum_{x_j\in\Lambda_j}2^{-j}(|a_{x_j,1}|+|a_{x_j,0}|+\|f_{x_j}\|_{L^\infty}) \bigg)(\sigma_j2^j\kappa)^{(1-\alpha)/12}.
\end{equation}
We thus need to appropriately bound the sum on the right-hand side.

Choosing $L=C2^{-j/2}$ (with $C$ chosen so that $L \geq 2 \|V\|_{L^1([s_{j+1}, s_j], L^\infty(\T^2))}$), after rescaling in time, Lemma~\ref{lem:box_bounds} implies that
\begin{align*}
\|f_{x_j}\|_{L^\infty}&\leq \Big\|\sol_{s_{j+1},t}^{V,\kappa,\T^d}\varphi_{j+1}\Big\|_{L^\infty([s_{j+1},s_j]\times (A_j+x_j))}
\\&\leq\|\varphi_{j+1}\|_{L^\infty(A_j+[-L,L]^2+x_j)}+C\|\varphi_{j+1}\|_{L^\infty(\T^2)}\exp(-C^{-1}(\sigma_j 2^j\kappa)^{-1}).    
\end{align*}
Since $|a_{x_j,1}|+|a_{x_j,0}|\leq 2\|\varphi_j\|_{L^\infty(A_j+[-L,L]^2+x_j)}$ and $\|\varphi_{j+1}\|_{L^\infty}\leq \|\varphi\|_{L^\infty}$, this implies that
\begin{align*}
&\sum_{x_j\in\Lambda_j}2^{-j}(|a_{x_j,1}|+|a_{x_j,0}|+\|f_{x_j}\|_{L^\infty})
\\&\qquad\leq 2^{-j}\sum_{x_j\in\Lambda_j}\|\varphi_{j+1}\|_{L^\infty(A_j+[-L,L]^2+x_j)}+C\|\varphi_{j+1}\|_{L^\infty}\exp(-C^{-1}(\sigma_j2^j\kappa)^{-1})
\\&\qquad\leq C\Big(\|\varphi_{j+1}\|_{L^1} +\|\varphi_{j+1}\|_{L^\infty}\exp(-C^{-1}(\sigma_j2^j\kappa)^{-1})\Big)
\\&\qquad\leq C\Big(\|\varphi\|_{L^1} +\|\varphi\|_{L^\infty}\exp(-C^{-1}(\sigma_j2^j\kappa)^{-1})\Big)
\end{align*}
where the second inequality follows by the definitions of $\varphi_{j+1}$ and $L$. 
With~\eqref{eq:j_to_j+1_bound_2} this shows that
\begin{equation}\label{eq:j_to_j+1_bound_3}
\Big\|\sol_{s_{j+1},s_j}^{V,\kappa,\T^2}\varphi_{j+1}-\varphi_{j}\Big\|_{L^1}\leq C\Big(\|\varphi\|_{L^1}+\|\varphi\|_{L^\infty}\exp\big({-C}(\sigma_j2^j\kappa)^{-1}\big)\Big)(\sigma_j2^j\kappa)^{(1-\alpha)/12}   
\end{equation}
for all $j$.

Using that $j\leq n$ and the definitions of $\sigma_j$ and $n$ we have the inequalities 
\begin{equation}
\label{eq:algebraic-rate-comp}
\sigma_j 2^j \kappa \leq \sigma_n 2^n \kappa \leq C2^{(\alpha+1)n/2} \kappa \leq C \kappa^{- \frac{3\alpha + 3}{2(\alpha+2)} +1} = C \kappa^{\frac{1-\alpha}{2(\alpha+2)}} \leq C \kappa^{(1-\alpha)/6}.\end{equation}
This implies that $\exp(-C^{-1}(\sigma_j2^j\kappa)^{-1})\leq \exp(-C^{-1} \kappa^{-(1-\alpha)/6})$ uniformly over $j$. Using this and inserting~\eqref{eq:j_to_j+1_bound_3} into~\eqref{eq:time_splits}, in total we find that
\[\big\|\sol_{s_n,1}^{V,\kappa,\T^2}\varphi-(\varphi)_{\T^2}\big\|_{L^1}\leq C\Big(\|\varphi\|_{L^1}+\|\varphi\|_{L^\infty}\exp(-C^{-1} \kappa^{-(1-\alpha)/6})\Big) \sum_{j=0}^{n-1}(\sigma_j2^j\kappa)^{(1-\alpha)/12}.\]
Since
\[\sum_{j=0}^{n-1}(\sigma_j2^j\kappa)^{(1-\alpha)/12}\leq C
(\sigma_n2^n\kappa)^{(1-\alpha)/12}\leq  C\kappa^{\frac{(1-\alpha)^2}{72}},\]
this concludes the proposition.
\end{proof}

The following proposition controls the evolution of the solution on the time interval $[\frac{1}{2},t_n]$. We show that the action of the drift and diffusion is small enough that, since the solution is starting as piecewise constant on boxes of sidelength $2^{-m/2}$, the solution remains (essentially) piecewise constant on the same boxes.

Proposition~\ref{prop:stays_near_constant} in some sense mirrors Proposition~\ref{prop:advection-noise-small} of Section~\ref{sec:two-cell}. As such---under the same considerations as in the choice of $n$ in Proposition~\ref{prop:2_cell_averaging}---we choose $m$ (essentially) as in Proposition~\ref{prop:advection-noise-small}.

\begin{proposition}\label{prop:stays_near_constant}
         There exists $C(\alpha)>0$ so that for all $\kappa>0$, letting
         \[n:=\big\lceil\tfrac{3}{2(\alpha+2)}\log_{\sqrt{2}}(\kappa^{-1})\big\rceil\text{ and } m :=\big\lceil\tfrac{\alpha+5}{4(\alpha+2)} \log_{\sqrt{2}}(\kappa^{-1})\big\rceil,\]
         if $\varphi\in L^1(\T^2)$ is constant on $A_m+x_m$ for all $x_m\in\Lambda_m$, then
\[\Big\|\sol^{V,\kappa,\T^2}_{1/2,s_n}\varphi-\varphi\Big\|_{L^1(\T^2)}\leq C\Big(\|\varphi\|_{L^1(\T^2)}+\|\varphi\|_{L^\infty(\T^2)}\exp\big({-C}\kappa^{- (1-\alpha)/6}\big)\Big) \kappa^{(1-\alpha)/12}.\]
\end{proposition}

The argument uses the following lemma as the main tool. The lemma states that constants are essentially unchanged (in an $L^1$ sense) provided the $L^1_t L^\infty_x$ norm of the flow and the diffusion are small enough.

\begin{lemma}\label{lem:constant_change} There exists $C>0$ so that for all vector fields $u\in L^1([0,1],L^\infty(B))$, boundary data $f\in L^\infty([0,1]\times\partial B)$, constants $a\in\R$, and $\kappa>0$, it holds that
\[\|\sol_{0,1}^{u,\kappa,f}a-a\|_{L^1(B)}\leq C\|f-a\|_{L^\infty([0,1]\times\partial B)}\big(\|u\|_{L^1([0,1],L^\infty(B))}+\kappa^{1/2}\log(\kappa^{-1})\big).\]
\end{lemma}

\begin{proof}[Proof of Proposition~\ref{prop:stays_near_constant}] We split the $L^1$ norm over boxes at scale $2^{-m/2}$
\begin{equation}\label{eq:small_box_split_2}
\Big\|\sol_{1/2,s_n}^{V,\kappa,\T^2}\varphi-\varphi\Big\|_{L^1}=\sum_{x_m\in \Lambda_m}\Big\|\sol_{1/2,s_n}^{V,\kappa,\T^2}\varphi-\varphi\Big\|_{L^1(A_m+x_m)}.    
\end{equation}
Letting $\tau:=s_n-\frac{1}{2}$, $f_{x_m}(t,x):=\sol_{1/2,(1/2+\tau t)}^{V,\kappa,\T^2}\varphi(\mathcal{R}^mx+x_m)$, $W(t,x):=\tau \mathcal{R}^{-m}V(\frac{1}{2}+\tau t,\mathcal{R}^mx+x_m)$, and $a_{x_m}$ equal the constant value $\varphi$ takes on $A_m+x_m$ for all $x_m\in\Lambda_m$, by rescaling in time and space it holds that
\[\sol_{1/2,s_n}^{V,\kappa,\T^2}\varphi(\mathcal{R}^mx+x_m)=\sol_{0,1}^{W,\tau2^m\kappa,f_{x_m}}a_{x_m}(x)\]
for all $x\in B$. Lemma~\ref{lem:constant_change} thus implies that
\begin{align}
\Big\|\sol_{1/2,s_n}^{V,\kappa,\T^2}\varphi-\varphi\Big\|_{L^1(A_m+x_m)}&=2^{-m}\Big\|\sol_{0,1}^{W,\tau2^m\kappa,f_{x_m}}a_{x_m}-a_{x_m}\Big\|_{L^1(B)}\notag
\\&\leq C2^{-m}(|a_{x_m}|+\|f_{x_m}\|_{L^\infty})\big(\|W\|_{L^1_tL^\infty_x}+(\tau2^m\kappa)^{1/2}\log\big((\tau2^m\kappa)^{-1}\big)\big).\label{eq:small_box_bound}
\end{align}
By the definition of $W$ and $V$ we have that
\[\|W\|_{L^1_tL^\infty_x}\leq2^{m/2}\sum_{j=n}^\infty \sigma_j\|V\|_{L^\infty([s_{j+1},s_j]\times \T^2)}\leq 2^{m/2}\sum_{j=n}^\infty 2^{-j/2}\|v\|_{L^\infty}\leq C 2^{(m-n)/2}\]
uniformly over $x_m.$  Since $\tau\leq C\sigma_n$,~\eqref{eq:algebraic-rate-comp} implies that
\begin{equation}
    \label{eq:algebraic-rate-comp-2}
\tau2^m\kappa\leq C\sigma_n 2^n\kappa\leq C\kappa^{(1-\alpha)/6}.
\end{equation}
We also have that
\[2^{(m-n)/2} \leq C \kappa^{-\frac{\alpha+5}{4(\alpha+2)} + \frac{3}{2(\alpha+2)}} = C \kappa^{\frac{1-\alpha}{4(\alpha+2)}} \leq C \kappa^{(1-\alpha)/12}.\]
Together these inequalities show that
\[\|W\|_{L^1_tL^\infty_x}+(\tau2^m\kappa)^{1/2}\log\big((\tau2^m\kappa)^{-1}\big)\leq C\kappa^{(1-\alpha)/12}.\]
Combining~\eqref{eq:small_box_split_2} and \eqref{eq:small_box_bound}, we have thus found that
\[\Big\|\sol_{1/2,s_n}^{V,\kappa,\T^2}\varphi-\varphi\Big\|_{L^1}\leq C\bigg(\sum_{x_m\in \Lambda_m}2^{-m}(|a_{x_m}|+\|f_{x_m}\|_{L^\infty})\bigg)\kappa^{(1-\alpha)/12}.\]
Thus, to conclude the proposition, it only remains to bound the sum on the right-hand side above.

First, we note that by the definition of $V$
\[\|V\|_{L^1([1/2,s_n],L^\infty(\T^2))}\leq \sum_{j=n}^\infty2^{-j/2}\|v\|_{L^\infty}\leq C2^{-n/2}.\]
After rescaling in time, Lemma~\ref{lem:box_bounds} with $L:=C2^{-m/2}\geq 2\|V\|_{L^1_tL^\infty_x}$ implies that
\[\|f_{x_m}\|_{L^\infty}\leq \|\sol^{V,\kappa, \T^2}_{1/2,t}\varphi\|_{L^\infty([1/2,s_n]\times A_m+x_m)}\leq \|\varphi\|_{L^\infty( A_m+[-L,L]^2+x_m)}+C\|\varphi\|_{L^\infty}\exp(-C^{-1}(\tau 2^m\kappa)^{-1}).\]
 Note also that $|a_{x_m}|\leq \|\varphi\|_{L^\infty( A_m+[-L,L]^2+x_m)}$, thus
\[|a_{x_m}|+\|f_{x_m}\|_{L^\infty}\leq  C\big(\|\varphi\|_{L^\infty( A_m+[-L,L]^2+x_m)}+\|\varphi\|_{L^\infty(\T^2)}\exp(-C^{-1}\kappa^{-(1-\alpha)/6})\big),\]
where we use~\eqref{eq:algebraic-rate-comp-2}. We thus find in total that
\begin{align*}
\sum_{x_m\in\Lambda_m}2^{-m}(|a_{x_m}|+\|f_{x_m}\|_{L^\infty})&\leq C2^{-m}\sum_{x_m\in\Lambda_m} \|\varphi\|_{L^\infty( A_m+[-L,L]^2+x_m)}+\|\varphi\|_{L^\infty(\T^2)}\exp\big({-C^{-1}\kappa^{-\frac{1-\alpha}{6}}}\big)
\\&\leq C\big(\|\varphi\|_{L^1}+\|\varphi\|_{L^\infty}\exp\big({-C^{-1}\kappa^{-(1-\alpha)/6}}\big)\big),
\end{align*}
which concludes the proposition.
\end{proof}  

Finally, we conclude the proof of Theorem~\ref{thm:total_dissipation} by combining Propositions~\ref{prop:2_cell_averaging},~\ref{prop:stays_near_constant}, and an initial smoothing estimate by pure diffusion on the time interval $[0,\frac{1}{2}].$

\begin{proof}[Proof of Theorem~\ref{thm:total_dissipation}] Since $V(t,x)=0$ when $t\in[0,1/2]$, we have that
\[\sol_{0,1/2}^{V,\kappa,\T^2}\theta_0=\Phi_{\kappa/2}*\theta_0.\]
Letting $n$ and $m$ be as in Propositions~\ref{prop:2_cell_averaging} and~\ref{prop:stays_near_constant}, we let
\[\varphi(x):=\sum_{x_m\in\Lambda_m}(\Phi_{\kappa/2}*\theta_0)_{A_m+x_m}\indc_{A_m+x_m}(x)\]
so that $(\varphi)_{\T^2}=(\Phi_{\kappa/2}*\theta_0)_{\T^2}=0$, and $\varphi$ satisfies the conditions of Propositions~\ref{prop:2_cell_averaging} and~\ref{prop:stays_near_constant}.

Using the triangle inequality and the fact that $\sol^{V,\kappa,\T^2}_{s,t}$ is an $L^1$ contraction, we have that
\begin{equation}\label{eq:time_split}
\Big\|\sol_{0,1}^{V,\kappa,\T^2}\theta_0\Big\|_{L^1}\leq\Big\|\sol_{0,1/2}^{V,\kappa,\T^2}\theta_0-\varphi\Big\|_{L^1}+\Big\|\sol_{1/2,s_n}^{V,\kappa,\T^2}\varphi-\varphi\Big\|_{L^1}+\Big\|\sol_{s_n,1}^{V,\kappa,\T^2}\varphi-(\varphi)_{\T^2}\Big\|_{L^1}.
\end{equation}
Propositions~\ref{prop:2_cell_averaging} and~\ref{prop:stays_near_constant} together imply that
\[\Big\|\sol_{1/2,s_n}^{V,\kappa,\T^2}\varphi-\varphi\Big\|_{L^1}+\Big\|\sol_{s_n,1}^{V,\kappa,\T^2}\varphi-(\varphi)_{\T^2}\Big\|_{L^1}\leq C\Big(\|\varphi\|_{L^1}+\|\varphi\|_{L^\infty}\exp\big({-C^{-1}}\kappa^{-(1-\alpha)/6}\big)\Big)\kappa^{\frac{(1-\alpha)^2}{72}}.\]
Since $\|\varphi\|_{L^1}\leq \|\Phi_{\kappa/2}*\theta_0\|_{L^1}\leq \|\theta_0\|_{TV}$ and $\|\varphi\|_{L^\infty}\leq\|\Phi_{\kappa/2}*\theta_0\|_{L^\infty} \leq C\kappa^{-1}\|\theta_0\|_{TV},$ it holds that
\[\|\varphi\|_{L^1}+\|\varphi\|_{L^\infty}\exp\big({-C^{-1}}\kappa^{-(1-\alpha)/6}\big)\Big)\leq C\|\theta_0\|_{TV}.\]
To conclude the theorem, it thus only remains to bound the first term on the right-hand side of~\eqref{eq:time_split}.

Using the Poincar\'e inequality,
\begin{align*}
\big\|\sol_{0,1/2}^{V,\kappa,\T^2}\theta_0-\varphi\big\|_{L^1}&\leq\sum_{x_m\in\Lambda_m}\Big\|\Phi_{\kappa/2}*\theta_0-\Big(\Phi_{\kappa/2}*\theta_0\Big)_{A_m+x_m}\Big\|_{L^1(A_m+x_m)}
\\&\leq C2^{-m/2} \sum_{x_m\in\Lambda_m}\|\nabla\Phi_{\kappa/2}*\theta_0\|_{L^1(A_m+x_m)}
\\&= C2^{-m/2}\|\nabla \Phi_{\kappa/2}*\theta_0\|_{L^1(\T^2)}
\\&\leq C2^{-m/2}\kappa^{-1/2}\|\theta_0\|_{TV}.
\end{align*}
By our choice of $m$
\[2^{-m/2}\kappa^{-1/2}\leq C\kappa^{\frac{\alpha+5}{4(\alpha+2)} - \frac{1}{2}} = C \kappa^{\frac{1-\alpha}{4(\alpha+2)}} \leq C \kappa^{(1-\alpha)/12}.\]
Combining the above estimates, we conclude.
\end{proof}

\appendix

\section{Stability estimates using stochastic flows}
\label{appendix:stability}

For convenience, throughout this section, we work on the square box $B=[0,1]^2$ and the square torus as opposed to Definition~\ref{def:box}. As the proofs are easily modified to handle non-isotropic diffusions, at the cost of slightly modifying the constant, the statements for the ``self-similar box'' $[0,\sqrt{2}] \times [0,1]$ follow.

We will repeatedly use the stochastic characteristic representation of drift-diffusion equations and the characteristic representation of transport equations. In particular, given a vector field $u$ in $L^\infty([0,1],L^\infty(B))$ we let $X^{\kappa}_t(x)$ solve the backwards stochastic differential equation
\begin{equation}\label{eq:stochastic_characteristic}\begin{cases}
    \frac{d}{dt}X^\kappa_t(x)=u(t,X^\kappa_t(x))\,dt+\sqrt{2\kappa}\,dw_{1-t},
    \\ X_1^\kappa(x)=x,
\end{cases}
\end{equation}
where $w_t$ is standard $2$-dimensional Brownian motion. That is, for all $t\in[0,1],$
\[X_t^\kappa(x)=x-\int_t^1 u(s,X_s^\kappa(x))\,ds+\sqrt{2\kappa}w_{1-t}.\]
Then, defining the stopping time $\tau_x:=\sup\Big\{t\in[0,1]\mid X_t^\kappa(x)\in\partial B\Big\}$ with respect to the backwards filtration $\mathcal{F}_t=\sigma(w_{1-s}:s\geq t)$, by the Feynman--Kac formula it holds that for all boundary data $f\in L^\infty([0,T]\times \partial B),$ the solution operator $\sol^{u,\kappa,f}_{0,1}$ to~\eqref{eq:drift_diff_equation} has the following representation:
\[\sol_{0,1}^{u,\kappa,f}\theta_0(x)=\mathbb{E}\big[\theta_0\big(X^{\kappa}_0(x)\big);\tau_x\leq 0\big]+\mathbb{E}\big[f\big(\tau_x,X^{\kappa}_{\tau_x}(x)\big);\tau_x>0\big].\]
This immediately implies that $\|\sol_{0,1}^{u,\kappa,f}\theta_0\|_{L^\infty}\leq \|\theta_0\|_{L^\infty}+\|f\|_{L^\infty}$. Similarly, the solution operator $\sol_{0,1}^{u,\kappa,\T^2}$ satisfies
\[\sol_{0,1}^{u,\kappa,\T^2}\theta_0(x)=\E\big[\theta_0\big(X_0^\kappa(x)\big)\big].\]

We also let $X_t(x)$ denote the backwards ODE given by setting $\kappa$ to $0$ in~\eqref{eq:stochastic_characteristic}. Then, by the characteristic representation of the transport equation, we similarly find that
\[\sol_{0,1}^u\theta_0(x)=\theta_0\big(X_0(x)\big),\]
and thus $\|\sol_{0,1}^u\theta_0\|_{L^\infty}\leq \|\theta_0\|_{L^\infty}.$

\begin{proof}[Proof of Lemma~\ref{lem:small_kappa_difference}]
We will exploit the characteristic representations to bound the difference between $\sol_{0,1}^{u,f,\kappa}\indc_E(x)$ and $\sol_{0,1}^{u}\indc_E(x)$ when $x$ is sufficiently far away from the boundary of $B$ and $E$. In particular, we define
\[A:=\big\{x\in B\mid \dist\big(X_0(x),\partial E \cup \partial B\big) + \dist\big(x, \partial B)> \kappa^{1/2}\log(\kappa^{-1})\big\}.\]
Then we have that 
\begin{align}
\notag
\big\|\sol_{0,1}^{u,\kappa,f}\indc_E-\sol_{0,1}^u\indc_E\big\|_{L^1(B)}&=\int_{B\setminus A} \big|\sol_{0,1}^{u,\kappa,f}\indc_E(x)-\sol_{0,1}^u\indc_E(x)\big|\,dx+\int_{A} \big|\sol_{0,1}^{u,\kappa,f}\indc_E(x)-\sol_{0,1}^u\indc_E(x)\big|\,dx
\\&\leq \big(1+\|f\|_{L^\infty}\big)|B\setminus A|+ \sup_{x \in A}|\sol_{0,1}^{u,f,\kappa}\indc_E(x)-\sol_{0,1}^{u}\indc_E(x)| .
\label{eq:characteristic-split}
\end{align}
Since $u$ is divergence-free,
\begin{align}
\notag
|B\setminus A|&\leq 2\bigg|\Big\{x\in B\mid \dist(x,\partial E\cup\partial B)\leq \kappa^{1/2}\log(\kappa^{-1})\Big\}\bigg|
\\&\leq C\big(1+|\partial E|\big)\kappa^{1/2}\log(\kappa^{-1})    
\label{eq:measure-curve-bound}
\end{align}
where we have used that for any $\eps>0$ and finite Lipschitz curve $S$
\[\big|\big\{x\mid \dist(x,S)\leq \eps\big\}\big|\leq C(\eps|S|+\eps^2).\]
Thus, by~\eqref{eq:characteristic-split} and the measure bound~\eqref{eq:measure-curve-bound}, in order to conclude, it suffices to show there exists $C\big(\|u\|_{L^\infty_tW^{1,\infty}_x}\big)>0$ so that
\[|\sol_{0,1}^{u,f,\kappa}\indc_E(x)-\sol_{0,1}^{u}\indc_E(x)|\leq C\big(1+\|f\|_{L^\infty}\big)\kappa,\]
for all $x\in A$ since $\kappa \ll\kappa^{1/2}\log(\kappa^{-1})$

We fix $x \in A$. Then, expanding the characteristic representations, we find that
\begin{align*}|\sol_{0,1}^{u,f,\kappa}\indc_E(x)-\sol_{0,1}^{u}\indc_E(x)|&\leq\Big|\mathbb{E}\big[\indc_E\big(X_0^\kappa(x)\big)-\indc_E\big(X_0(x)\big); \tau_x \leq 0\big] \Big|
\\&\qquad+\Big|\mathbb{E}\big[f\big(\tau_x,X^{\kappa}_{\tau_x}(x)\big) -\indc_{E}\big(X_0(x)\big);\tau_x>0\big]\Big|.
\end{align*}
Since $\indc_E$ is equal to $0$ or $1$, we have that
\begin{equation}\label{eq:membership_bound}\Big|\mathbb{E}\big[\indc_E\big(X_0^\kappa(x)\big)-\indc_E\big(X_0(x)\big) ; \tau_x \leq 0\big] \Big|\leq\P\Big(\indc_E\big(X_0^\kappa(x)\big)\neq\indc_E\big(X_0(x)\big)\Big).\end{equation}
For the other term, we bound
\[\Big|\mathbb{E}\big[f\big(\tau_x,X^{\kappa}_{\tau_x}(x)\big) -\indc_{E}\big(X_0(x)\big);\tau_x>0\big]\Big|\leq \big(1+\|f\|_{L^\infty}\big)\P(\tau_x>0).
\]
We thus need to show that~\eqref{eq:membership_bound} and $\P(\tau_x>0)$ are both bounded by $C\kappa$.

For~\eqref{eq:membership_bound} we note that $A$ is defined so that if
\[|X_0^\kappa(x)-X_0(x)|\leq \kappa^{1/2}\log(\kappa^{-1}),\]
then $\indc_{E}\big(X_0^\kappa(x)\big)=\indc_E\big(X_0(x)\big)$. Therefore
\[\P\Big(\indc_E\big(X_0^\kappa(x)\big)\neq\indc_E\big(X_0(x)\big)\Big)\leq \P\Big(\big|X_0^\kappa(x)-X_0(x)\big|>\kappa^{1/2}\log(\kappa^{-1})\Big).\]
Using the equations for $X_t^\kappa(x)$ and $X_t(x)$,
\[|X_t^\kappa(x)-X_t(x)|\leq \|u\|_{L^\infty_tW^{1,\infty}_x}\int_t^1|X_s^\kappa(x)-X_s(x)|\,ds+\sqrt{2\kappa}\sup_{s\in[t,1]}|w_{1-s}|,\]
thus Gr\"onwall's inequality implies that there exists $C(\|u\|_{L^\infty_tW^{1,\infty}_x})>0$ so that 
\[|X_0^\kappa(x)-X_0(x)|\leq C\kappa^{1/2}\sup_{s\in[0,1]}|w_s|.\]
Using the well-known bound on the probability that Brownian motion has large excursions, we find that
\[\P\Big(\big|X_0^\kappa(x)-X_0(x)\big|>\kappa^{1/2}\log(\kappa^{-1})\Big)\leq \P\bigg(\sup_{s\in[0,1]}|w_s|>C^{-1}\log(\kappa^{-1})\bigg) \leq C e^{-C^{-1} \log(\kappa^{-1})^2} \leq  C\kappa,
\]
as desired.

To conclude, we just need to bound $\P(\tau_x>0)$. It is here that we use that $u$ is tangential to $\partial B$. We first show that there exists $C\big(\|u\|_{L^\infty_t W^{1,\infty}_x}\big)>0$ so that for all $x\in A$
\[\{\tau_x>0\}\subseteq \bigg\{\sup_{[0,1]}|w_t|\geq C^{-1}\log(\kappa^{-1})\bigg\}.\]
We then find that $\P(\tau_x>0)\leq C\kappa$ using the same Brownian excursion bound as used above.

Suppose that $\tau_x>0$ so that $X_{\tau_x}^\kappa(x)\in\partial B.$  Without loss of generality, we may assume that $X_{\tau_x}^\kappa(x)\in \{0\}\times [0,1].$ That is, $X_{\tau_x}^\kappa(x)$ is in the left boundary of $B$. Then, letting $Y_t:=X_{\tau_x+t}^\kappa(x)$, we have that
\[Y_t=Y_0+\int_0^t u(\tau_x+s,Y_s)\,ds-\sqrt{2\kappa}(w_{1-\tau_x+t}-w_{1-\tau_x}),\]
and $Y_{1-\tau_x}=x$. Since $u_1(0,z)=0$ for all $z\in[0,1]$ and $u$ is uniformly Lipschitz in time, this implies that
\[|(Y_t)_1|\leq \|u\|_{L^\infty W^{1,\infty}}\int_0^t|(Y_s)_1|\,ds+\sup_{s\in[0,t]}\sqrt{2\kappa}|w_{1-\tau_x+s}-w_{1-\tau_x}|,\]
where $(Y_t)_1$ denotes the first coordinate of $Y_t$. Gr\"onwall's inequality thus implies that there exists $C\big(\|u\|_{L^\infty_t W^{1,\infty}_x}\big)>0$ so that
\[|x_1|=|(Y_{1-\tau_x})_1|\leq C\kappa^{1/2}\sup_{t\in[0,\tau_x]}|w_{1-\tau_x+t}-w_{1-\tau_x}|.\]
Since $x\in A$, it must be the case that $|x_1|\geq \kappa^{1/2}\log(\kappa^{-1})$, thus
\[\kappa^{1/2}\log(\kappa^{-1})\leq C\kappa^{1/2}\sup_{t\in[0,\tau_x]}|w_{1-\tau_x +t}-w_{1-\tau_x}|\leq C\kappa^{1/2}\sup_{t\in[0,1]} |w_t|,\]
    concluding the bound on $\P(\tau_x>0)$ and thus the proof.
\end{proof}

\begin{proof}[Proof of Lemma~\ref{lem:mean-preservation}] 
We prove instead the equivalent statement,
    \[\Big|\int_B \sol_{0,1}^{u,\kappa,f} \theta_0- \theta_0\,dx\Big| \leq C\big(\|\theta_0\|_{L^\infty(B)}+\|f\|_{L^\infty([0,1]\times\partial B)}\big)\big(\|u\|_{L^1([0,1],L^\infty(B))}+\kappa^{1/2} \log (\kappa^{-1})\big).\]

Throughout we let $\delta:=\|u\|_{ L^1_tL^\infty_x}+\kappa^{1/2}\log(\kappa^{-1})$ and $B_\delta:=[\delta,1-\delta]^2.$ Then using the characteristic representations of $\sol^{u,\kappa,f}_{0,1}$ and $\sol^u_{0,1}$ we find that
\[\int_B\sol^{u,\kappa,f}_{0,1}\theta_0(x)\,dx=\int_{B_\delta}\mathbb{E}\Big[\theta_0\big(X^{\kappa}_0(x)\big);\tau_x\leq 0\Big]+\mathbb{E}\Big[f\big(\tau_x,X^{\kappa}_{\tau_x}(x)\big);\tau_x>0\Big]\,dx+\int_{B\setminus B_\delta} \sol^{u,\kappa,f}_{0,1}\theta_0(x)\,\,dx.\]
Since $|B\setminus B_\delta|\leq C\delta$, we have that
\begin{equation}\label{eq:B_delta_bound}\bigg|\int_{B\setminus B_\delta} \sol^{u,\kappa,f}_{0,1}\theta_0(x)\,\,dx\bigg|\leq C\big(\|\theta_0\|_{L^\infty}+\|f\|_{L^\infty}\big)\delta.\end{equation}
We also bound
\begin{equation}\label{eq:stopping_holder_bound}\bigg|\int_{B_\delta}\mathbb{E}\Big[f\big(\tau_x,X^{\kappa}_{\tau_x}(x)\big);\tau_x>0\Big]\,dx\bigg|\leq \|f\|_{L^\infty}\int_{B_\delta}\mathbb{P}(\tau_x>0)\,dx.
\end{equation}
Combining~\eqref{eq:B_delta_bound} and~\eqref{eq:stopping_holder_bound} in total we have that
\begin{align}
\bigg|\int_B\sol^{u,\kappa,f}_{0,1}\theta_0(x)\,dx-\int_B\theta_0(x)\,dx\bigg|&\leq \bigg|\mathbb{E}\bigg[\int_{B_\delta}\mathbb{E}\Big[\theta_0\big(X^{\kappa}_0(x)\big);\tau_x\leq 0\Big]\,dx-\int_B\theta_0(x)\,dx\bigg]\bigg|\label{eq:transport_term}
\\&\quad+\|f\|_{L^\infty}\int_{B_\delta}\mathbb{P}(\tau_x>0)\,dx\label{eq:stopping_term}
\\&\quad+C\big(\|\theta_0\|_{L^\infty}+\|f\|_{L^\infty}\big)\delta.\notag
\end{align} To conclude we need to bound~\eqref{eq:transport_term} and~\eqref{eq:stopping_term}.

Before continuing, we will give good bounds on the probability that $\tau_x>0$ uniformly over $B_\delta$. First, we note that for all $t\geq \tau_x$
\[|X_t^\kappa(x)-x|\leq \|u\|_{ L^1_tL^\infty_x}+\sqrt{2\kappa}\sup_{t\in[0,1]}|w_t|.\]
This implies that
\begin{equation}\label{eq:stopping_bound}
\P\bigg(\sup_{x\in B_\delta}\tau_x>0\bigg)\leq \mathbb{P}\bigg(\sup_{t\in[0,1]}|w_t|\geq C^{-1}\log(\kappa^{-1})\bigg)\leq C\kappa
\end{equation}
for some $C>0,$ bounding the excursions of Brownian motion as in the proof of Lemma~\ref{lem:small_kappa_difference}. As a consequence, we immediately find that~\eqref{eq:stopping_term} is bounded by $C\|f\|_{L^\infty}\kappa$.

To bound~\eqref{eq:transport_term} first we note that
\begin{align*}&\bigg|\mathbb{E}\bigg[\int_{B_\delta}\mathbb{E}\Big[\theta_0\big(X^{\kappa}_0(x)\big);\tau_x\leq 0\Big]\,dx-\int_B\theta_0(x)\,dx\bigg]\bigg|
\\&\qquad\leq \bigg|\mathbb{E}\bigg[\int_{B_\delta}\mathbb{E}\Big[\theta_0\big(X^{\kappa}_0(x)\big);\sup_{x\in B_\delta}\tau_x\leq 0\Big]\,dx-\int_B\theta_0(x)\,dx\bigg]\bigg|+\|\theta_0\|_{L^\infty}\mathbb{P}\bigg(\sup_{x\in B_\delta}\tau_x>0\bigg)
\\&\qquad\leq \mathbb{E}\bigg[\bigg|\int_{B_\delta}\theta_0\big(X^{\kappa}_0(x)\big)\,dx-\int_B\theta_0(x)\,dx\bigg|;\sup_{x\in B_\delta}\tau_x\leq 0\bigg] +C\|\theta_0\|_{L^\infty}\kappa,
\end{align*}
where in the last line we've used Fubini and~\eqref{eq:stopping_bound}. When $\sup_{x\in B_\delta}\tau_x\leq 0$, then letting $\tilde B_\delta:=X_0^{\kappa}(B_\delta)$, since $u$ is divergence-free it holds that $\tilde B_\delta\subseteq B$,
\begin{align*}
\int_{B_\delta}\theta_0\big(X^{\kappa}_0(x)\big)\,dx=\int_{\tilde B_\delta}\theta_0(x)\,dx,
\end{align*}
and $|B\setminus \tilde B_\delta|\leq C\delta$. This implies that
\[\bigg|\int_{B_\delta}\theta_0\big(X^{\kappa}_0(x)\big)\,dx-\int_B\theta_0(x)\,dx\bigg|\leq \|\theta_0\|_{L^\infty}|B\setminus \tilde B_\delta|\leq C\|\theta_0\|_{L^\infty}\delta.\]
In total we have thus found that~\eqref{eq:transport_term} is bounded by $C\|\theta_0\|_{L^\infty}(\delta+\kappa).$ Combining this with the bound on~\eqref{eq:stopping_term} and the fact that $\kappa<\delta$, this concludes the lemma.
\end{proof}

\begin{proof}[Proof of Lemma~\ref{lem:close-to-toroidal-heat}] Letting $f(t,x):=\sol_{0,t}^{0,\kappa,\T^2}\theta_0(x)$, it is clear that $\|f\|_{L^\infty([0,1]\times\partial B)}\leq \|\theta_0\|_{L^\infty(B)}$ and $\sol_{0,1}^{0,\kappa,\T^2}\theta_0=\sol_{0,1}^{0,\kappa,f}\theta_0$. It thus suffices to prove that for any $f,g\in L^\infty([0,1]\times \partial B)$ and $\theta_0\in L^\infty(\T^2)$
\[\|\sol^{0,\kappa,f}_{0,1}\theta_0-\sol^{0,\kappa,g}_{0,1}\theta_0\|_{L^1}\leq C\big(\|f-g\|_{L^\infty}\big)\kappa^{1/2}\log(\kappa^{-1}).\]

Using the characteristic representation of $\sol^{0,\kappa,f}_{0,1}\theta_0$ and $\sol^{0,\kappa,g}_{0,1}\theta_0$
\[\sol^{0,\kappa,f}_{0,1}\theta_0(x)-\sol^{0,\kappa,g}_{0,1}\theta_0(x)=\mathbb{E}\Big[f\big(\tau_x,X_{\tau_x}^\kappa\big)-g\big(\tau_x,X_{\tau_x}^\kappa\big);\tau_x>0\Big].\]
Letting $B_\kappa:=[\kappa^{1/2}\log(\kappa^{-1}),1-\kappa^{1/2}\log(\kappa^{-1})]^2$
it thus holds that
\begin{align}\notag
\|\sol^{0,\kappa,f}_{0,1}\theta_0-\sol^{0,\kappa,g}_{0,1}\theta_0\|_{L^1}&\leq \|f-g\|_{L^\infty}\bigg(\int_{B_\kappa}\P(\tau_x>0)\,dx+|B\setminus B_\kappa|\bigg)
\\&\leq \|f-g\|_{L^\infty}\bigg(\int_{B_\kappa}\P(\tau_x>0)\,dx+C\kappa^{1/2}\log(\kappa^{-1})\bigg).
\label{eq:split_bound}
\end{align}
Since $X_t^\kappa(x)=x+\sqrt{2\kappa}w_{1-t}$ for $t\geq \tau_x$ we have that for all $x\in B_\kappa$
\[\P(\tau_x>0)\leq \P\bigg(\sup_{t\in[0,1]}|w_t|>C^{-1}\log(\kappa^{-1})\bigg)\leq C\kappa.\]
Inserting this into~\eqref{eq:split_bound}, we conclude the lemma.
\end{proof}

\begin{proof}[Proof of Lemma~\ref{lem:box_bounds}]
Suppose $x\in [-\ell,\ell]^2$. Then using the stochastic 
characteristic representation of $\sol_{0,1}^{u,\kappa,\T^2}\theta_0(x)$,
\[\sol_{0,1}^{u,\kappa,\T^2}\theta_0(x)=\mathbb{E}\Big[\theta_0\big(X^{\kappa}_0(x)\big)\Big].\]
We thus have that
\begin{align}
\Big|\sol_{0,1}^{u,\kappa,\T^2}\theta_0(x)\Big|&\leq \bigg|\mathbb{E}\Big[\theta_0\big(X^{\kappa}_0(x)\big);X_0^\kappa(x)\in[-\ell-L,\ell+L]^2\Big]\bigg|
\notag\\&\qquad\qquad+\bigg|\mathbb{E}\Big[\theta_0\big(X^{\kappa}_0(x)\big);X_0^\kappa(x)\notin[-\ell-L,\ell+L]^2\Big]\bigg|\notag
\\&\leq \|\theta_0\|_{L^\infty([-\ell-L,\ell+L]^2)}+\|\theta_0\|_{L^\infty}\mathbb{P}\Big(X_0^\kappa(x)\notin[-\ell-L,\ell+L]^2\Big).\label{eq:expectation_split}
\end{align}
Since
\[\big|X_0^\kappa(x)-x\big|\leq \|u\|_{L^1_tL^\infty_x}+\sqrt{2\kappa}\sup_{t\in[0,1]}|w_t|,\]
$x\in[-\ell,\ell]^2,$ and $L>2\|u\|_{L^1_tL^\infty_x}$
\[\mathbb{P}\Big(X_0^\kappa(x)\notin[-\ell-L,\ell+L]^2\Big)\leq \mathbb{P}\Big(\sup_{t\in[0,1]}|w_t|>C^{-1}\kappa^{-1/2}L\Big)\leq C\exp\big(-C^{-1}\kappa^{-1}L^2\big).\]
Combining this bound with~\eqref{eq:expectation_split}, this concludes the lemma.
\end{proof}

\begin{proof}[Proof of Lemma~\ref{lem:constant_change}]
Note that
\[\|\sol_{0,1}^{u,\kappa,f}a-a\|_{L^1(B)} = \|\sol_{0,1}^{u,\kappa,f-a}0\|_{L^1(B)}.\]
Then by the stochastic characteristic representation
\[|\sol_{0,1}^{u,\kappa,f-a}0(x)|=|\E[f(\tau_x,X_{\tau_x})-a;\tau_x>0]| \leq  \|f-a\|_{L^\infty}\P(\tau_x>0).\]
Thus
\[\|\sol_{0,1}^{u,\kappa,f}a-a\|_{L^1(B)} \leq \|f-a\|_{L^\infty}\int_{B}\P(\tau_x>0)\,dx.\]
Letting $\delta$ and $B_\delta$ be as in Lemma~\ref{lem:mean-preservation}, it holds that 
\[\int_{B}\P(\tau_x>0)\,dx\leq |B\setminus B_\delta|+\int_{B_\delta}\P(\tau_x>0)\,dx\leq C\big(\|u\|_{L^1_tL^\infty_x}+\kappa^{1/2} \log (\kappa^{-1})\big)\]
using the same bounds on $\P(\tau_x>0)$. This concludes the lemma.
\end{proof}

\section{Regularity of the vector fields}
\label{appendix:regularity}

\begin{proof}[Proof of Lemma~\ref{lem:v-props}]
    Items~\ref{item:v-gluable} and~\ref{item:v-periodic} are direct from Items~\ref{item:dies} and~\ref{item:periodic} of Theorem~\ref{thm:ACM_field}. For Item~\ref{item:v-regular}, we first compute the spatial regularity. For $t\in [t_n,t_{n+1}]$, we have by definition and then interpolating that
    \begin{align*}\|v(t,\cdot)\|_{C^\alpha} 
    &= \tau_n^{-1}\|U(\tau_n^{-1}(t-t_n) +n,\cdot)\|_{C^\alpha} 
    \\&\leq C 5^{(1-\alpha)n} \|U(\tau_n^{-1}(t-t_n) +n,\cdot)\|_{L^\infty}^{1-\alpha} \|U(\tau_n^{-1}(t-t_n) +n,\cdot)\|_{W^{1,\infty}}^\alpha 
    \\&\leq C 5^{(1-\alpha)n} 5^{-(1-\alpha)n} = C,
    \end{align*}
    where for the final inequality, we use Item~\ref{item:L^infty_bound} of Theorem~\ref{thm:ACM_field}. Taking the supremum over $t \in [0,1]$, we conclude $v \in L^\infty([0,1], C^\alpha(B))$. 

    For the time regularity, let $k := \floor{\frac{\alpha}{1-\alpha}}$ and $\gamma := \frac{\alpha}{1-\alpha} - \floor{\frac{\alpha}{1-\alpha}}.$ Using Item~\ref{item:L^infty_bound} of Theorem~\ref{thm:ACM_field}, we have that for $t_n \leq s \leq t \leq t_{n+1}$, that 
    \begin{align*}
    \|\partial_t^k v(t,\cdot) - \partial_t^k v(s,\cdot)\|_{L^\infty} &\leq C 5^{(1-\alpha) (k+1) n} \| \partial_t^kU(\tau_n^{-1}(t-t_n) +n,\cdot) - \partial_t^kU(\tau_n^{-1}(s-t_n) +n,\cdot)\|_{L^\infty}
    \\&\leq C 5^{(1-\alpha)(k+1)n} \|\partial_t^k U\|_{C^\gamma([n,n+1], L^\infty(B))} 5^{(1-\alpha)\gamma n} |t-s|^{\gamma}
    \\&\leq C 5^{(1-\alpha)(\frac{\alpha}{1-\alpha} + 1) n - n} |t-s|^\gamma = C |t-s|^\gamma.
    \end{align*}
    If $t_n \leq s \leq t_{n+1} \leq t \leq t_{n+2}$, using the above, we get
    \begin{align*} \|\partial_t^k v(t,\cdot) - \partial_t^k v(s,\cdot)\|_{L^\infty} &\leq \|\partial_t^k v(t,\cdot) - \partial_t^k v(t_{n+1},\cdot)\|_{L^\infty} + \| \partial_t^k v(t_{n+1},\cdot)- \partial_t^k v(s,\cdot)\|_{L^\infty}
    \\&\leq C (|t - t_{n+1}|^\gamma + |s- t_{n+1}|^\gamma) \leq C |t-s|^\gamma.
    \end{align*}
    Finally, if $0 \leq  t_n \leq s \leq  t_{n+2} \leq t$, we have that
    \[
    \|\partial_t^k v(t,\cdot) - \partial_t^k v(s,\cdot)\|_{L^\infty} \leq  \|\partial_t^k v(t,\cdot)\|_{L^\infty} + \|\partial_t^k v(s,\cdot)\|_{L^\infty}
    \leq  C 5^{-n} \leq C 5^{-n} |t-s|^\gamma \tau_n^{-\gamma} \leq C |t-s|^\gamma.
    \]
    Thus combining the cases, we get the time regularity, $v \in C^{\frac{\alpha}{1-\alpha}}([0,1], L^\infty(B)),$ as desired.
 \end{proof}

{\small
\bibliographystyle{alpha}
\bibliography{references}
}

\end{document}